\RequirePackage{fix-cm}
\documentclass[smallextended]{svjour3}       % onecolumn (second format)
\smartqed  % flush right qed marks, e.g. at end of proof

\usepackage{amsfonts,comment}
\usepackage{amssymb,amsmath,url}
\usepackage{mathrsfs}
\usepackage{hyperref}
\usepackage{graphicx,color,float,epstopdf}
\usepackage{mathrsfs}
\usepackage{pst-blur,pstricks-add}
\usepackage{cases}
\usepackage{algorithm}
\usepackage{algorithmicx}
\usepackage{algpseudocode}
\usepackage{graphics,booktabs,epsfig}
\usepackage[numbers,sort&compress]{natbib}
\usepackage{booktabs}
\usepackage{subcaption}
\usepackage{multirow,marvosym}
\usepackage[numbers,sort&compress]{natbib}

\DeclareMathOperator{\prox}{Prox}
\DeclareMathOperator{\supp}{supp}

\DeclareMathOperator*{\dom}{dom}

\newtheorem{assumption}{Assumption}
\newcommand{\h}[1]{\mathbf{#1}}

\DeclareMathOperator{\sign}{sign}

\newcommand{\x}{\mathbf{x}}
\newcommand{\y}{\mathbf{y}}
\newcommand{\z}{\mathbf{z}}
\newcommand{\g}{\mathbf{g}}

\newcommand{\M}{\mathcal{M}}

\newcommand{\nn}{\nonumber}

\newcommand{\R}{\mathbb{R}}

\def\cB{{\mathscr{B}}}

\begin{document}
	
	\title{Bridging Identification and Second-Order Acceleration: A Fast Alternating Minimization Framework for Composite Optimization}

	\author{Zihao Xia\and  Min Tao}

	\institute{Z. Xia \at School of Mathematics, Nanjing University, Nanjing, 210093, China.\\
		\Letter\ M. Tao \at School of Mathematics, National Key Laboratory for Novel Software Technology, Nanjing University, Nanjing, 210093, Republic of China.
		\href{mailto:taom@nju.edu.cn}{taom@nju.edu.cn}  \\}
	\date{Received: date / Accepted: date}
	
	\maketitle
\begin{abstract}
We consider a class of composite optimization problems involving a smooth function and a proper, lower semicontinuous regularizer, which may be nonconvex and nonsmooth. 
			We propose a novel alternating minimization framework that integrates proximal-gradient steps with cubic-regularized Newton updates restricted to a dynamically identified low-dimensional subspace. 
			Under the Kurdyka--Łojasiewicz (KL) property, we establish global convergence of the proposed method to a stationary point. 
			Moreover, by incorporating an adaptive thresholding strategy guided by the KL exponent, we prove a finite identification property without imposing any nondegeneracy assumptions. 
			We further develop a local convergence analysis and show that the proposed method attains a worst-case iteration complexity of $\mathcal{O}(\varepsilon^{-3/2})$ for achieving approximate second-order stationarity. 
			Numerical experiments on both synthetic and real datasets demonstrate the efficiency  of the proposed framework.
\end{abstract}

\keywords{Nonconvex composite optimization, Subspace  cubic
			regularized Newton, Identification, Proximal gradient, Global convergence}

\subclass{
90C26 \and 65K05\and 90C30}

\section{Introduction}
		In this paper, we consider the following problem:
		\begin{equation}
			\min_{\mathbf{x}\in\mathbb{R}^n} F(\mathbf{x}) = g(\mathbf{x}) + h(\mathbf{x}),\label{Problem}
		\end{equation}
		where $g:\mathbb{R}^{n}\to \overline{\mathbb{R}} := \mathbb{R} \cup \{
			+\infty\}$ is proper, lower semicontinuous, possibly nonconvex and nonsmooth, and \emph{proximal-friendly}, i.e., the proximal mapping $\prox_{\gamma g}$ can be computed efficiently,
		and  $h:\mathbb{R}^{n}\to\mathbb{R}$ is twice continuously differentiable with $L$-Lipschitz continuous gradient for some $L\ge 0$.
		We impose the following assumptions:
		\begin{assumption}\label{Ass1}
			\begin{enumerate}
				\item[(i)] For any $\mathbf{x}\in\mathbb{R}^{n}$, define the manifold
				\[
				\mathcal{M} := \{\mathbf{z}\in\mathbb{R}^{n} : \supp(\mathbf{z}) = \supp(\mathbf{x})\},
				\]
				and let $\mathcal{S} \subseteq \supp(\mathbf{x})$. The restricted Hessian $\nabla^2_{\mathcal{S}} F$ is strictly continuous at $\mathbf{x}$ relative to $\mathcal{M}$ with modulus $\mathrm{lip}_{\mathcal{M}}\nabla^2_{\mathcal{S}}F(\mathbf{x})$.
				
				\item[(ii)] The function $F$ is bounded below, i.e., $\underline{F} := \inf_{\mathbf{x}\in\mathbb{R}^{n}} F(\mathbf{x}) > -\infty$.
			\end{enumerate}
		\end{assumption}
		Composite optimization problems of the form \eqref{Problem} arise ubiquitously in modern applications, including sparse signal recovery, machine learning, image reconstruction, and statistical estimation. In these settings, the nonsmooth term $g$ is typically employed to promote desirable structures such as sparsity, low-rankness, or group sparsity, while the smooth component $h$ represents data fidelity or loss functions derived from observations. These applications often involve large-scale and ill-conditioned data, making it essential to design efficient algorithms that can exploit the composite structure while handling nonconvexity and nonsmoothness.

		%	Assumption \ref{Ass1}(i)  means that the restriction Hessian is locally Lipschitz continuous around $\x$. %It leads to for any $M(\x) > \mathrm{lip}\nabla^2_{\mathcal{S}}F({\x})$, there exists a $\delta >0$ such that
		%	\begin{eqnarray}\label{mk}
			%		\|\nabla^2_{\mathcal{S} }F(\x_1)-\nabla^2_{\mathcal{S} }F(\x_2)\|\le M(\x)\|\x_1-\x_2\|,
			%		\quad \forall \x_1,\x_2\in \cB(\x,\delta)\cap \mathcal{M}_{\mathcal{S}}.
			%	\end{eqnarray}
		%First-order methods have been studied well for solving the composite problems of the form \eqref{Problem}. 
		Proximal gradient-type algorithms and their variants have been extensively studied for solving the composite problems of the form \eqref{Problem} \citep{{beck2017first}}.
		Second-order methods have been recognized as another potential approach  for solving nonconvex optimization problems due to their fast local convergence. 
		The cubic regularized Newton (CRN) method, originally proposed by \citep{nesterov2006cubic}, has attracted significant attention. Under the assumption that the Hessian is globally Lipschitz continuous, CRN achieves the optimal worst-case iteration complexity of order $\mathcal{O}(\varepsilon^{-3/2})$ for finding an approximate second-order stationary point. This optimal complexity has motivated a series of works aimed at improving its practical efficiency.
		\citep{cartis2011adaptive,cartis2011adaptive2} proposed the Adaptive Regularization with Cubics (ARC) framework, which adaptively updates the regularization parameter and approximately minimizes a cubic model at each iteration. By exploiting Krylov subspace techniques, ARC achieves favorable scalability while preserving the optimal $\mathcal{O}(\varepsilon^{-3/2})$ complexity. Subsequent works have further improved the efficiency of solving the cubic subproblem, including accelerated first-order approaches \citep{jiang2021accelerated}, convex reformulations \citep{jiang2022cubic}, and momentum-based variants \citep{wang2020cubic}. In addition, the asymptotic convergence behavior of CRN has been analyzed under the  Kurdyka-\L ojasiewicz (KL) framework \citep{zhou2018convergence}, revealing refined local convergence properties.
		
		However,  classical cubic regularization methods  require that the objective function is twice continuously differentiable on $\mathbb R^n$,
		they cannot be applied to  (\ref{Problem}) directly.  They are always computationally demanding in high-dimensional settings due to requiring full gradients and Hessians. These limitations motivate us to develop subspace second-order methods that restrict to a low-dimensional subspace using locally second-order curvature information without the assumption of the objective function twice continuously differentiable.
		%broadly categorized according to how the subspace is selected.

		\subsection{Related works} \citep{bareilles2023newton} proposed a framework that alternates between proximal gradient steps and 
		Riemannian Newton steps restricted to an identified manifold. Under suitable assumptions, the algorithm identifies the active manifold in finite iteration and subsequently achieves superlinear or quadratic convergence. Similar ideas have been explored for sparse optimization problems with nonconvex regularization. In particular, \citep{wu2023regularized} developed a hybrid proximal gradient and subspace Newton method, proving finite support identification and establishing global convergence under the KL property, along with local superlinear convergence under an error bound condition. More recently, \citep{tao2024partlysmooth} studied $L_1/L_2$ regularization and proposed a two-stage framework that first identifies the active manifold and then applies a semismooth Newton method, achieving superlinear and even quadratic convergence.
		Another line is based on optimality conditions to choose subspace. For example, \citep{chen2017reduced} introduced a subspace Newton method for $L_1$-regularized convex problems using the KKT residual to guide subspace selection.  \citep{wang2024alternating} proposed the IReNA method for nonconvex sparse optimization with separable concave regularizers. Their approach alternates between reweighted $L_1$ steps and subspace Newton updates, and under the nondegeneracy and KL assumptions, they established finite identification and local quadratic convergence.
		More recently, randomized subspace strategies have been proposed to reduce computational costs. Works such as \citep{zhao2024cubic,chen2022accelerating} develop randomized variants of cubic regularization methods, showing that with sufficiently large sample sizes, the optimal $\mathcal{O}(\varepsilon^{-3/2})$ complexity can be achieved with high probability.

		\subsection{Motivation and contributions}
		Although these methods demonstrate that restricting second-order updates to a subspace can significantly improve efficiency, existing approaches typically rely on either strong structural assumptions (such as partly smooth, prox-regularity and nondegeneracy conditions), or heuristic subspace selection rules whose theoretical justification remains incomplete. Moreover, most existing frameworks separate the \emph{identification phase} and the \emph{second-order acceleration phase}, leading to multi-stage. 
		This raises the following fundamental question:

			{\bf{\emph{Can we design a unified algorithmic framework that simultaneously performs subspace identification and second-order acceleration, while enjoying global convergence guarantees?}}}

		\noindent
		%\textbf{Contributions.}
		In this paper, we provide an affirmative answer to this question by proposing a novel subspace cubic Newton-type framework that seamlessly integrates first-order identification and second-order refinement within a unified iterative scheme. The central idea is to dynamically construct an adaptive low-dimensional subspace based on the current iterate, and to perform a cubic-regularized Newton step restricted to this subspace.
		Our main contributions can be summarized as follows:
		\begin{itemize}
			\item We propose a unified alternating minimization framework that integrates proximal-gradient steps with cubic-regularized Newton updates restricted to a low-dimensional subspace. This unified framework bridges identification and second-order acceleration under weak
			structure.
			
			\item We establish global convergence of the proposed method to a first-order stationary point under the Kurdyka--Łojasiewicz (KL) property.
			
			\item We show that the algorithm identifies the active subspace (or manifold) in a finite number of iterations via an adaptive thresholding strategy driven by the KL exponent of the objective function, without any nondegeneracy conditions. Once identification occurs, we establish fast local convergence, characterized by a local iteration complexity of $\mathcal{O}(\varepsilon^{-3/2})$ for computing an $(\varepsilon_g,\varepsilon_H)$-approximate second-order stationary point.
		\end{itemize}
		\vspace{0.5em}
		\noindent
		\textbf{Organization.}
		The remainder of this paper is organized as follows. 
		In Section~\ref{sec:prelim}, we introduce the problem formulation and necessary preliminaries. 
		Section~\ref{sec:algorithm} presents the proposed algorithm and its main properties. 
		Global convergence is established in Section~\ref{sec:global}. 
		In Section~\ref{sec:finite}, we further show that the finite identification property holds for the proposed algorithm without requiring a nondegeneracy condition. 
		Section~\ref{sec:local} is devoted to the local convergence analysis. 
		Numerical experiments are reported in Section~\ref{sec:numerical}, followed by concluding remarks in Section~\ref{sec:conclusion}.
		
		\section{Preliminary}\label{sec:prelim}
		
		For any $r>0$, let $\cB({\widehat{\mathbf{x}}}, r)$ represent the Euclidean ball centered at ${\widehat{\mathbf{x}}}$ with radius $r$, that is, $\cB({\widehat{\mathbf{x}}},r)=\{\mathbf{x}:\|\mathbf{x}-{\widehat{\mathbf{x}}}\|\leq r\}$.
		For a function $f: \mathbb{R}^{n}\rightarrow\mathbb{R}$,  dom$(f) = \{\mathbf{x}: f(\mathbf{x}) < \infty\}$ represents its effective domain. $f$ is proper if and only if dom$(f) \neq \varnothing$ and $f(\mathbf{x}) > -\infty$ for all $\mathbf{x}\in$ dom$(f)$. We say $f$ is lower-semicontinuous if $f(\overline{\mathbf{x}})\le\lim\inf_{{\mathbf{x}}\to \overline{\mathbf{x}}}f(\mathbf{x})$ for $\overline{\mathbf{x}}\in\text{dom}(f)$. The symbol $\sharp(\mathcal{D})$ denotes the cardinality of $\mathcal{D}$. 
		$\lambda_{\min}(\cdot)$ denotes the minimum  eigenvalue. $\mathbb S^{n\times n}$ represents $n\times n$ symmetric matrix.
		%For $\overline{\mathbf{x}}\in\text{dom}(f)$, the set:
%		\begin{equation*}
%			{\widehat\partial} f({\overline{\mathbf{x}}}):=\left\{{\mathbf{v}} \in {\mathbb R}^n :{\liminf \limits_{\mathbf{x}\to{\overline{\mathbf{x}}}\;{\mathbf{x}}\neq {\overline{\mathbf{x}}}}} \frac{f(\mathbf{x})-f(\overline{\mathbf{x}})-\langle {\mathbf{v}},{\mathbf{x}}-{\overline{\mathbf{x}}}\rangle}{\|{\mathbf{x}}-\overline{\mathbf{x}}\|}\ge 0\right\}
%		\end{equation*}
%		is the $Fr\acute{e}chet$ subdifferential of $f$. The limiting subdifferential of $f$ at $\overline{\mathbf{x}}$ is defined as:
%		\begin{equation*}
%			\partial f({\overline{\mathbf{x}}}):=\left\{{\mathbf{v}} \in \mathbb{R}^n \! : \! \exists\; {\mathbf{x}}^k \rightarrow {\overline{\mathbf{x}}}, \ f({\mathbf{x}}^k)\rightarrow f(\overline{\mathbf{x}}), {\mathbf{v}}^k  \rightarrow{\mathbf{v}} \ \text{as} \ k \rightarrow +\infty, \
%			{\mathbf{v}}^k\in{\widehat\partial} f({\mathbf{x}}^k) \right\}.
%		\end{equation*}
		\noindent For a proper, lower-semicontinuous function $f$, the proximal operator with module $\alpha$ is defined as:
		$
		\prox_{\alpha f}(\mathbf{z}) = \mathop{\arg\min}_{\mathbf{x}\in\mathbb{R}^{n}}\left(  f(\mathbf{x}) + \frac{1}{2\alpha}\|\mathbf{x} - \mathbf{z}\|^2\right).
		$
		%	Let $B(\x,r)$ denotes the open ball with center $\x\in\mRn$ and radius $r > 0$. The definition of prox-regular function is then as follows:
		%	\begin{Def}[Prox-Regular]
			%		A function $f: \mathbb{R}^{n}\rightarrow\overline{\mathbb{R}}$ is called prox-regular at $\overline{\x}$ for $\overline{\bv}$ if $f$ is finite and locally lower semicontinuous at $\overline{\x}$, $\overline{\bv}\in\partial f(\overline{\x})$ and there exist $\delta > 0$ and $\rho\ge 0$ such that for all $\x,\y\in B(\x,r)$ and $\bv\in\partial f(\x)$ with $f(\x) \le f(\overline{\x}) + \delta$ and $||\bv - \overline{\bv}||\le\delta$, we have
			%		\begin{equation}
				%			f(\y) \ge f(\x) + \langle \bv, \y - \x\rangle - \frac{\rho}{2}||\y - \x||^2
				%		\end{equation}
			%		The function $f$ is called prox-regular at $\overline{\x}$, if it is prox-regular at $\overline{\x}$ for all $\overline{\bv}\in\partial f(\overline{\x})$.
			%	\end{Def}
		We recall the Kurdyka-\L ojasiewicz (KL) property and the KL exponent  \citep{attouch2009convergence}, and the notion of strict continuity \citep[Definition 9.1]{rockafellar1998variational}.
		\begin{definition} 
			\label{def:KL}
			A proper function $h: \mathbb R^n\rightarrow {\overline{\mathbb R}}$ is said to satisfy the {\it Kurdyka-\L ojasiewicz (KL) property} at a point ${\widehat{\mathbf{x}}}\in {\text{\rm dom }}\partial h$ if there exist a constant $\eta\in(0,+\infty]$, an open neighborhood
			$U$ of ${\widehat{\mathbf{x}}},$ and a concave and continuous function $\phi:\;[0,\eta)\rightarrow[0,+\infty)$ with the properties
			\begin{itemize}
				\item[(i)] $\phi(0)=0$;
				\item[(ii)] $\phi$ is continuously differentiable on $(0,\eta)$ and $\phi'{(\cdot)}>0$ on $(0,\eta)$;
				\item[(iii)] for every $\mathbf{x}\in U$ with $h({\widehat{\mathbf{x}}})< h(\mathbf{x})< h(\widehat{\mathbf{x}}) +\eta$, it holds
				\begin{equation*}
					\phi'(h(\mathbf{x})-h(\widehat{\mathbf{x}})){\text{\rm{dist}}}(\mathbf{0},\partial h({\mathbf{x}}))\ge 1. 
				\end{equation*}
			\end{itemize}
			The class of so-called desingularization functions $\phi:[0,\eta)\rightarrow[0,+\infty)$ fulfilling (i)-(iii) is denoted by $\Phi_\eta$.
		\end{definition}

		\begin{definition}\label{def:KL_exponent}
			Let $f:\mathbb{R}^n\to(-\infty,+\infty]$ be proper and lower semicontinuous, and
			let $\overline \x\in\mathrm{dom}(\partial f)$ be a point satisfying $0\in\partial f(\overline \x)$.
			We say that $f$ satisfies the \emph{KL property at  $\overline \x$ with exponent
				$\theta\in[0,1)$} if there exist constants $\eta>0$, $\varepsilon>0$ and $\mu>0$ such that
			\begin{equation*}\label{eq:KL_basic}
				\mathrm{dist}({{\bf 0}},\partial f(\x))\ge\mu\,(f(\x)-f(\overline \x))^{\theta},
			\end{equation*}
			whenever $\|\x-\overline \x\|<\varepsilon$ and
			$0<f(\x)-f(\overline \x)<\eta$.
			%		If $f$ satisfies \eqref{eq:KL_basic} at every point in $\mathrm{crit}\,f$, we say that
			%		$f$ has the KL property.
		\end{definition}
		
		\begin{definition}\label{def:Lipschitz}
			Let $F$ be a single-valued mapping defined on a set $D \subseteq \mathbb{R}^{n}$, with values in $\mathbb{R}^{m \times s}$. Let $X \subseteq D$. We say that $F$ is {\it strictly continuous} at $\overline{\x}$ relative to $X$ if $\overline{\x} \in X$ and the value
			\begin{equation*}
				\mathrm{lip}_{X} F(\overline{\x}) := \limsup_{\substack{\x,\x' \to \overline{\x} \\ \x,\x' \in X,\, \x \neq \x'}} 
				\frac{\|F(\x') - F(\x)\|}{\|\x - \x'\|}
			\end{equation*}
			is finite. In this case, $\mathrm{lip}_{X} F(\overline{\x})$ is called the Lipschitz modulus of $F$ at $\overline{\x}$ relative to $X$.
		\end{definition}
		
		Throughout the paper, we denote by $\mathrm{lip}_{\mathcal{M}} \nabla^2 F(\x)$ the Lipschitz modulus of $\nabla^2 F$ at $\x$ relative to the set
		$
		\mathcal{M}$.
		When there is no ambiguity, we abbreviate it as $\mathrm{lip}\,\nabla^2 F(\x)$.
		At the end of this section, we provide the closed-form characterization of the global minimizers of the cubic Newton subproblem, as shown in~\citep[Corollary~8.3.1]{cartis2022evaluation}, and present the second-order necessary conditions; see~\citep[Proposition~1]{nesterov2006cubic}.

		\begin{proposition}\label{thm:cubic_solution}
			 Let  $\widehat H\in\mathbb S^{n\times n}$ and $\mathbf{d}^*$ be any global solution of the following cubic-regularized Newton problem:
			\begin{equation}
				\label{prob:cubic_newton_sub}
				\mathbf{d}^*\in
				\mathop{\arg\min}_{\mathbf{d}\in\mathbb{R}^{n}}
				\left\{
				\langle \g, \mathbf{d}\rangle
				+
				\frac{1}{2}\langle \widehat H \mathbf{d},\mathbf{d} \rangle
				+
				\frac{\sigma}{6}\|\mathbf{d}\|^3
				\right\}.
			\end{equation}
			Let $\lambda_1$ denote the smallest eigenvalue of $\widehat H$
			and  $\mathbf{v}_1$ be any eigenvector associated with $\lambda_1$.
			Then for any global minimizer $\mathbf{d}^*$, there exists a scalar $\lambda^\star$ such that $\lambda^\star \ge -\lambda_1$ and
			\begin{equation*}\label{eq:cubic-sol}
				(\widehat H+\lambda^\star I)\mathbf{d}^*=-\mathbf{g},
				\qquad
				\lambda^\star=\frac{\sigma}{2}\|\mathbf{d}^*\|.
			\end{equation*}
			Moreover,
			\[
			\mathbf{d}^*=
			\begin{cases}
				-(\widehat H+\lambda^\star I)^{-1}\mathbf{g},
				& \text{if }\lambda^\star>-\lambda_1,\\[4pt]
				-(\widehat H+\lambda^\star I)^{\dagger}\mathbf{g}+\alpha\,\mathbf{v}_1,
				& \text{if }\lambda^\star=-\lambda_1,
			\end{cases}
			\]
			and $\alpha\in\mathbb{R}$ is one of the two roots of
			$
			\bigl\|
			-(\widehat H+\lambda^\star I)^{\dagger}\mathbf{g}+\alpha\,\mathbf{v}_1
			\bigr\|
			=
			\displaystyle{\frac{\lambda^\star}{\sigma/2}},
			$ and $(\cdot)^{\dagger}$ represents the generalized inverse.
		\end{proposition}

		\begin{proposition}
			\label{secondorder}
			Let $\widehat H$ be defined in Proposition \ref{thm:cubic_solution} and $\mathbf{d}^*$ be any solution of (\ref{prob:cubic_newton_sub}),
			then
			$
			\displaystyle{\widehat H + \frac{\sigma}{2}\|\mathbf{d}^*\| I \succcurlyeq 0}.$
		\end{proposition}

		\section{A Framework Bridging Identification and Second-Order Acceleration}\label{sec:algorithm}
		\begin{algorithm}[!htpb]
			\caption{Alternating Minimization via Subspace Cubic Newton-Proximal Gradient (SCN-PG)}
			\label{alg:apgcn}
			\begin{algorithmic}[1]
				\Require{$\mathbf{y}^{0}\in\R^{n}$,  $0<\alpha<1/L$, parameters $c>0$, $\eta>0$, $0<q<1$, $\omega_k\geq\underline{\omega}>1$,  $\varepsilon_k> 0$,  {\tt MaxIt}.}
				
				\For{$k=1,2,\ldots,{\texttt{MaxIt}}$}
				
				\State \textbf{(Proximal--gradient step)}
				%\begin{equation*}
				%	\label{PGstep}
				$\mathbf{x}^{k}\in\prox_{\alpha g}
					\bigl(\mathbf{y}^{k-1}-\alpha\nabla h(\mathbf{y}^{k-1})\bigr).$
				%\end{equation*}
				\State \textbf{($\varepsilon$-Active subspace identification)}
				%\begin{equation*}
					%\label{eq:subspace_def}
					$\mathcal S^{k}:=\{\,i:\ |x_i^{k}|\ge \varepsilon_k\,\}, \; n_k:=\sharp(\mathcal{S}^{k}).$
				%\end{equation*}
				\If{$\mathcal{S}^k \neq\emptyset$}
				\State Let 
				%\begin{equation*}\label{gk} 
					$\g^{k}=\nabla_{\mathcal S^{k}}F(\mathbf{x}^{k}),\;
					H^{k}=\nabla^{2}_{\mathcal S^{k}}F(\mathbf{x}^{k}),$
				%	\end{equation*}
				%\State and 
				$M(\x^k) \geq \omega_k\mathrm{lip}\nabla^2_{\mathcal{S}^k}F({\x^k}),\; \sigma_k\ge \frac{2(M(\mathbf{x}^k)+c)}{3}.$
				\State \textbf{(Subspace cubic--regularized Newton step)}
				\begin{equation}\label{subd}\mathbf{d}^{k}\in\mathop{\arg\min}_{\mathbf{d}\in\R^{n_k}}
					\left\{
					\langle \g^{k},\mathbf{d}\rangle
					+\frac12\langle (H^{k}+\eta I)\mathbf{d},\mathbf{d}\rangle
					+\frac{\sigma_k}{6}\|\mathbf{d}\|^{3}
					\right\}.
				\end{equation}			
				\State \textbf{(Backtracking line search)}
				Set $\beta_{k,0}=1$.
				\For{$j_k=0,1,2,\ldots$}
				\State Define $
				\mathbf{y}^{k,j_k}_{\mathcal S^{k}}=\mathbf{x}^{k}_{\mathcal S^{k}}+\beta_{k,j_k}\mathbf{d}^{k},
				\;
				\mathbf{y}^{k,j_k}_{(\mathcal S^{k})^{c}}=\mathbf{x}^{k}_{(\mathcal S^{k})^{c}}.$
				\If{$
					F(\mathbf{y}^{k,j_k})
					\le
					F(\mathbf{x}^{k})
					-\frac{\eta}{4}\|\mathbf{y}^{k,j_k}-\mathbf{x}^{k}\|^{2}
					-\frac{c}{6}\|\mathbf{y}^{k,j_k}-\mathbf{x}^{k}\|^{3}
					$}
				\State Set $\mathbf{y}^{k}=\mathbf{y}^{k,j_k},\beta_k = \beta_{k,j_k}$ and \textbf{break}.
				\Else
				\State $\beta_{k,j_k+1}=q\,\beta_{k,j_k}$.
				\EndIf
				\EndFor
				\Else
				\State{Set $\y^k = \x^k$.}
				\EndIf
				
				\If{termination criterion satisfied}
				\State {break}
				\EndIf
				
				\EndFor
				
				\State {Output:} $\mathbf{y}^{k}$.
			\end{algorithmic}
		\end{algorithm}
		
		In this section, we propose an alternating minimization framework that combines proximal-gradient updates with subspace cubic-regularized Newton refinement. We refer to this method as the \emph{Subspace Cubic Newton--Proximal Gradient} (SCN-PG) method; see Algorithm~\ref{alg:apgcn}.
		At each iteration, a proximal-gradient step generates a candidate point $\mathbf{x}^k$. Although a constant stepsize is adopted for simplicity, it can be replaced by a backtracking strategy to adaptively select a suitable stepsize.
		Based on $\mathbf{x}^k$, an $\varepsilon_k$-active subspace $\mathcal S^k$ is identified, consisting of indices whose magnitudes exceed a prescribed threshold. The choice of $\varepsilon_k$ will be specified in Section~\ref{sec:finite}. This step aims to identify a low-dimensional manifold on which the objective function exhibits improved smoothness properties.
		Restricted to the subspace indexed by $\mathcal S^k$, a cubic-regularized Newton step is computed by approximately minimizing a local cubic model constructed from the restricted gradient $\g^k$ and restricted Hessian $H^k$. The cubic regularization parameter $\sigma_k$ ensures global stability and well-posedness of the subproblem, while an additional parameter $\eta>0$ provides further control on the step.
		A backtracking line search is then performed along the direction $\mathbf{d}^k$ defined in~\eqref{subd}, restricted to the subspace $\mathcal S^k$, to guarantee sufficient decrease of the objective function. The acceptance criterion
		%\[
%		F(\mathbf{y}^{k,j_k})
%		\le
%		F(\mathbf{x}^{k})
%		-\frac{\eta}{4}\|\mathbf{y}^{k,j_k}-\mathbf{x}^{k}\|^{2}
%		-\frac{c}{6}\|\mathbf{y}^{k,j_k}-\mathbf{x}^{k}\|^{3}
%		\]
		incorporates both quadratic and cubic decrease terms, which play a key role in establishing global convergence. If the identified subspace is empty, the method reduces to a standard proximal-gradient step.
		
		Overall, the proposed algorithm integrates active subspace identification with second-order modeling, enabling efficient exploitation of local smoothness while preserving global convergence guarantees through proximal-gradient updates and cubic regularization.
		%At each iteration, a constant stepsize \(\alpha \in (0,1/L)\) is employed in the proximal-gradient step. Although a backtracking strategy can be used to eliminate the dependence on the Lipschitz constant \(L\), we adopt a constant stepsize for simplicity, and omit the extension to adaptive backtracking rules.
		
		%The subspace identification strategy is defined by selecting the index set
		%	\begin{equation}\label{eq:subspace_def}
			%		\mathcal{S}^{k}
			%		=
			%		\left\{\, i : |{\mathbf{x}}_i^{k}|\ge \varepsilon_{k} \,\right\},
			%		\qquad
			%		n_{k} = \sharp(\mathcal{S}^{k}).
			%	\end{equation}
		%	The restricted gradient $\mathbf{g}^{k}$ and the restricted Hessian $H^{k}$  are defined in (\ref{gk}).
		%\[
		%	\mathbf{g}^{k} = \nabla_{\mathcal{S}^{k}} F(\mathbf{x}^{k}),
		%	\qquad
		%	H^{k} = \nabla_{\mathcal{S}^{k}}^{2} F(\mathbf{x}^{k}).
		%	\]
		
		%For getting an estimation of the parameter of \(M(\mathbf{x}^k)\), we have several strategies:

		%Below, we establish the convergence and complexity analysis of SCN-PG. The overall idea is as follows: we first use the sequence $\{\mathbf{x}_k\}$ to guarantee the global convergence of the algorithm, and then analyze how the sequence $\{\mathbf{y}_k\}$ contributes to the acceleration effect.
		
		\section{Convergence Analysis}\label{sec:global}
		In this section, we establish the global convergence result of Algorithm~\ref{alg:apgcn}.
		\subsection{Descent property}
		
		First, we provide a variant of \citep[Lemma~1]{nesterov2006cubic}.
		
		\begin{lemma}\label{lem:inexact_cubic_upper}
			%Suppose that $F$ defined in \eqref{Problem} satisfies Assumption~\ref{Ass1}.
			Let $\x,\ \mathbf{d}\in\R^n$ and  $\M:=\{\z\in\R^{n}:\supp(\z)=\supp(\x) \}$, and $\supp(\mathbf{d})\subseteq \supp(\mathbf{x})$ satisfies $ \supp(\mathbf{x} + t\mathbf{d}) = \supp(\x)$ for all $t\in[0,1]$.
			Let  $\mathcal{S}$ be an index set satisfying $\supp(\mathbf{d})\subseteq\mathcal{S}\subseteq \supp(\mathbf{x})$, and there exist constants $M\ge 0$ and $a\ge 0$ such that for all  $t\in[0,1]$, the function $F$ satisfies
			\begin{equation}\label{eq:hess_var}
				\bigl\|\nabla_{\mathcal{S}}^2 F(\x+t\mathbf{d})-\nabla_{\mathcal{S}}^2 F(\x)\bigr\|
				\le M\,t\|\mathbf{d}\|+a.
			\end{equation}
			Then, for all $t\in[0,1]$, it holds that
			\begin{equation}\label{ineq:grad_var_bound}
				\bigl\|\nabla_{\mathcal{S}}F(\x+t\mathbf{d})-\nabla_{\mathcal{S}}F(\x)
				- t\,\nabla_{\mathcal{S}}^2F(\x)\mathbf{d}_{\mathcal{S}}\bigr\|
				\le a\,t\|\mathbf{d}\|+\frac{M}{2}\,t^2\|\mathbf{d}\|^2 .
			\end{equation}
			
			\noindent Moreover,
			\begin{equation}\label{ineq:cubic_upper}
				F(\x+\mathbf{d})\le F(\x)
				+\langle \nabla_{\mathcal{S}}F(\x),\mathbf{d}_{\mathcal{S}}\rangle
				+\frac12\langle \nabla_{\mathcal{S}}^2F(\x)\mathbf{d}_{\mathcal{S}},\mathbf{d}_{\mathcal{S}}\rangle
				+\frac{a}{2}\|\mathbf{d}\|^2+\frac{M}{6}\|\mathbf{d}\|^3.
			\end{equation}
		\end{lemma}
		
		\begin{proof}
			For $t\in[0,1]$ and define
			$
			\psi(t):=\nabla_{\mathcal{S}}F(\x+t\mathbf{d})\in\R^{\sharp(\mathcal{S})}.
			$
			Since $\supp(\x+t\mathbf{d})=\supp(\x)$ for all $t\in[0,1]$,  $\supp(\mathbf{d})\subseteq\mathcal{S}\subseteq \supp(\mathbf{x})$
			and (\ref{eq:hess_var}) holds for $t\in[0,1]$,
			$\psi$ is differentiable on $[0,1]$, with
			\[
			\psi'(t)=\nabla_{\mathcal{S}}^{2}F(\x+t\mathbf{d})\,\mathbf{d}_{\mathcal{S}}, \qquad t\in[0,1].
			\]
			Then, we obtain
			\begin{equation*}
				\psi(t)-\psi(0)-t\psi'(0)
				=\int_{0}^{t}\bigl(\psi'(u)-\psi'(0)\bigr)\,du 
				=\int_{0}^{t}\bigl(\nabla_{\mathcal{S}}^{2}F(\x+u\mathbf{d})-\nabla_{\mathcal{S}}^{2}F(\x)\bigr)
				\,\mathbf{d}_{\mathcal{S}}\,du .
			\end{equation*}
			Taking norms and using $\|\mathbf{d}_{\mathcal{S}}\| = \|\mathbf{d}\|$, together with
			\eqref{eq:hess_var}, yields
			\begin{equation*}
					\begin{aligned}
					\bigl\|\psi(t)-\psi(0)-t\psi'(0)\bigr\|
					&\le \int_{0}^{t}\bigl\|\nabla_{\mathcal{S}}^{2}F(\x+u\mathbf{d})
					-\nabla_{\mathcal{S}}^{2}F(\x)\bigr\|\,\|\mathbf{d}\|\,du \le \int_{0}^{t}\bigl(Mu\|\mathbf{d}\|+a\bigr)\|\mathbf{d}\|\,du \\
					&= a\,t\|\mathbf{d}\|+\frac{M}{2}t^{2}\|\mathbf{d}\|^{2},
				\end{aligned}
			\end{equation*}
			which proves \eqref{ineq:grad_var_bound}.
			Next, define $\phi(t)=F(\x+t\mathbf{d})$. Then, $\phi$ is
			differentiable on $[0,1]$ and satisfies
			$
			\phi'(t)=\langle \nabla_{\mathcal{S}}F(\x+t\mathbf{d}),\mathbf{d}_{\mathcal{S}}\rangle.
			$
			Consequently,
			\begin{equation*}
			\begin{aligned}
				&F(\x+\mathbf{d})-F(\x)-\langle \nabla_{\mathcal{S}}F(\x),\mathbf{d}_{\mathcal{S}}\rangle = \phi(1) - \phi(0) - \phi'(0)\\
				& = \int_{0}^{1}(\phi'(t) - \phi'(0))\,dt=\int_{0}^{1}\bigl\langle \nabla_{\mathcal{S}}F(\x+t\mathbf{d})
				-\nabla_{\mathcal{S}}F(\x),\mathbf{d}_{\mathcal{S}}\bigr\rangle\,dt \\
				&=\int_{0}^{1}\Bigl\langle t\nabla_{\mathcal{S}}^{2}F(\x)\mathbf{d}_{\mathcal{S}},
				\mathbf{d}_{\mathcal{S}}\Bigr\rangle\,dt
				+\int_{0}^{1}\Bigl\langle \nabla_{\mathcal{S}}F(\x+t\mathbf{d})
				-\nabla_{\mathcal{S}}F(\x)
				-t\nabla_{\mathcal{S}}^{2}F(\x)\mathbf{d}_{\mathcal{S}},\mathbf{d}_{\mathcal{S}}\Bigr\rangle\,dt \\
				&\le\frac 12\bigl\langle \nabla_{\mathcal{S}}^{2}F(\x)\mathbf{d}_{\mathcal{S}},\mathbf{d}_{\mathcal{S}}\rangle+\int_{0}^{1}\Bigl|\Bigl\langle \nabla_{\mathcal{S}}F(\x+t\mathbf{d})
				-\nabla_{\mathcal{S}}F(\x)
				-t\nabla_{\mathcal{S}}^{2}F(\x)\mathbf{d}_{\mathcal{S}},\mathbf{d}_{\mathcal{S}}\Bigr\rangle\Bigr|\,dt \\
				&\le\frac12\bigl\langle \nabla_{\mathcal{S}}^{2}F(\x)\mathbf{d}_{\mathcal{S}},\mathbf{d}_{\mathcal{S}}\rangle+ \int_{0}^{1}\bigl\|\nabla_{\mathcal{S}}F(\x+t\mathbf{d})
				-\nabla_{\mathcal{S}}F(\x)
				-t\nabla_{\mathcal{S}}^{2}F(\x)\mathbf{d}_{\mathcal{S}}\bigr\|\,\|\mathbf{d}_{\mathcal{S}}\|\,dt \\
				&\le \frac12\bigl\langle \nabla_{\mathcal{S}}^{2}F(\x)\mathbf{d}_{\mathcal{S}},\mathbf{d}_{\mathcal{S}}\rangle+\int_{0}^{1}\Bigl(a\,t\|\mathbf{d}\|
				+\frac{M}{2}t^{2}\|\mathbf{d}\|^{2}\Bigr)\|\mathbf{d}\|\,dt\\
				&=\frac12\bigl\langle \nabla_{\mathcal{S}}^{2}F(\x)\mathbf{d}_{\mathcal{S}},\mathbf{d}_{\mathcal{S}}\rangle+\frac{a}{2}\|\mathbf{d}\|^{2}+\frac{M}{6}\|\mathbf{d}\|^{3}.
			\end{aligned}
			\end{equation*}
			
			\noindent Combining the above estimates establishes \eqref{ineq:cubic_upper} and completes the
			proof.\qed
		\end{proof}

The following Theorem shows that the line search for $\beta_{k}$ in Algorithm \ref{alg:apgcn} will terminate in finite step.
\begin{theorem}
\label{Newton_decrease}
Assume that $F$ defined in \eqref{Problem} satisfies Assumption~\ref{Ass1}.
At iteration $k$, suppose that $\mathcal{S}^k\neq\emptyset$ and  $ c>0$
\begin{equation}\label{sigmak}
\sigma_{k} \;\geqslant\; \frac{2(M(\mathbf{x}^{k})+c)}{3}.
\end{equation} Then, the procedure of finding the smallest $j_k \in\{0,1,2,\ldots\}$ such that 
\begin{equation}\label{back}
F(\mathbf{y}^{k,j_k})
\le
F(\mathbf{x}^{k})
-\frac{\eta}{2}\|\mathbf{y}^{k,j_k}-\mathbf{x}^{k}\|^{2}
-\frac{c}{6}\|\mathbf{y}^{k,j_k}-\mathbf{x}^{k}\|^{3}
\end{equation}
is executed in a finite number of times, and so Algorithm \ref{alg:apgcn} is well-defined. 
%Consequently, there exists a finite $\widehat j_k$ such that $\mathbf{y}^{k}=\mathbf{y}^{k,\widehat j_k},\beta_k = \beta_{k,\widehat j_k}$, i.e.,
%the backtracking step of Line 9-16 will
%terminate in a finite iterations.
\end{theorem}

\begin{proof}
Let $\M^k:=\{\z\in\R^{n}:\supp(\z)=\supp(\x^k)\}$. For
the $M(\x^k)$ defined in Algorithm \ref{alg:apgcn}, there exists $\delta_k >0$ such that 
\begin{equation}\label{mmk}
\bigl\|
\nabla^2_{\mathcal{S}^{k}}F(\mathbf{x}_1)
-
\nabla^2_{\mathcal{S}^{k}}F(\mathbf{x}_2)
\bigr\|
\le
M(\mathbf{x}^k)\,
\|\mathbf{x}_1-\mathbf{x}_2\|,
\;
\forall \mathbf{x}_1,\mathbf{x}_2 \in \cB(\x^k,\delta_k)\cap \mathcal{M}^k.\\
\end{equation}
%due to $M(\x^k) \geq \omega_k\mathrm{lip}\nabla^2_{\mathcal{S}^k}F({\x^k})$, and $\omega_k\geq\underline{\omega}>1$.

Since $\supp(\mathbf{d}^{k})  \subseteq \supp(\x^{k})$, there exists sufficiently small $\widehat{\beta}_k\in(0,1]$ such that $\supp(\x^{k}+\beta\mathbf{d}^{k})=\supp(\x^{k})$ and $\x^{k}+\beta\mathbf{d}^{k}\in\cB(\mathbf{x}^k,\delta_k)\cap \mathcal{M}^k$ for all $\beta\in(0, \widehat{\beta}_k]$.
Invoking Lemma~\ref{lem:inexact_cubic_upper} and (\ref{mmk}), we obtain that
\begin{equation}
\label{eq:beta_upper}
F(\x^k+\beta \mathbf{d}^k)
\le
F(\x^{k})
+ \beta\langle \g^{k},\mathbf{d}^{k}\rangle
+ \frac{\beta^{2}}{2}\langle H^{k}\mathbf{d}^{k},\mathbf{d}^{k}\rangle
+ \frac{M(\x^{k})\beta^{3}}{6}\|\mathbf{d}^{k}\|^{3}.
\end{equation}

\noindent Invoking the optimality conditions of (\ref{subd}) and Proposition \ref{secondorder},
\begin{subequations} \label{eq:opt1}
	\begin{numcases}{\hbox{\quad}}
	\label{gsub}\g^{k} + (H^{k}+\eta I)\h d^{k}
+ \frac{\sigma_{k}}{2}\|\h d^{k}\|\h d^{k} = 0,\\[0.0cm]
	\label{Hsub}H^{k}+\eta I + \frac{\sigma_{k}}{2}\|\h d^{k}\| I \succeq 0.
	\end{numcases}
\end{subequations}

%\begin{eqnarray}\label{eq:opt1}\left\{\begin{array}{l}
%		\displaystyle{\g^{k} + (H^{k}+\eta I)\mathbf{d}^{k}
%			+ \frac{\sigma_{k}}{2}\|\mathbf{d}^{k}\|\mathbf{d}^{k} = 0}, \\[0.1cm]
%		\displaystyle{H^{k}+\eta I + \frac{\sigma_{k}}{2}\|\mathbf{d}^{k}\| I \succeq 0}.
%	\end{array}\right.
%\end{eqnarray}
Taking the inner product of the first equality in \eqref{gsub} with $\mathbf{d}^{k}$ yields
\begin{equation*}
\label{eq:first-lb}
\langle \g^{k},\ \mathbf{d}^{k}\rangle
=
-\langle (H^{k}+\eta I)\mathbf{d}^{k},\ \mathbf{d}^{k}\rangle
-\frac{\sigma_{k}}{2}\|\mathbf{d}^{k}\|^{3}.
\end{equation*}
Moreover, (\ref{Hsub}) leads to
$
\langle H^{k}\mathbf{d}^{k},\mathbf{d}^{k}\rangle
\geqslant
-\eta\|\mathbf{d}^{k}\|^{2}
-\frac{\sigma_{k}}{2}\|\mathbf{d}^{k}\|^{3}.$
\noindent Substituting these two  into \eqref{eq:beta_upper}, we obtain
%\begin{equation}
	\begin{eqnarray}
		&&F(\x^{k}+\beta\mathbf{d}^{k})
		\le
		F(\x^{k})
		+ \Big(\frac{\beta^{2}}{2}-\beta\Big)\langle H^{k}\mathbf{d}^{k},\ \mathbf{d}^{k}\rangle
		- \eta\beta\|\mathbf{d}^{k}\|^{2}- \Big(\frac{\sigma_{k}\beta}{2}- \frac{M(\x^{k})\beta^{3}}{6}\Big)\|\mathbf{d}^{k}\|^{3}\nn\\
		&&\le F(\x^k) - \frac{\eta\beta^2}{2}\|\mathbf{d}^{k}\|^{2} -  \Big(\frac{\sigma_{k}\beta^2}{4}- \frac{M(\x^{k})\beta^{3}}{6}\Big)\|\mathbf{d}^{k}\|^{3}\nn\\
		&&\le
		F(\x^{k})
		- \frac{\eta\beta^2}{2}\|\mathbf{d}^{k}\|^{2}
		- \frac{c\beta^3}{6}\|\mathbf{d}^{k}\|^{3}.\label{key11}
	\end{eqnarray}
%\end{equation}
\noindent Thus, there exists a finite $\widehat j_k$ such that $\beta_k = \beta_{k,\widehat j_k}\in (0,\widehat{\beta}_k]$,
and (\ref{back}) holds.
It implies that the backtracking step of Line 9-16 will terminate in finite many iterations, and Algorithm \ref{alg:apgcn} is well-defined. \qed
\end{proof}

%	\begin{Thm}
%		Suppose $\{X_k\}$ is bounded, then every accumulation point of $\{X_{k}\}$ is a stationary point for (\ref{matrix_problem}).
%	\end{Thm}
%	\begin{proof}
%		Suppose $X_{k_j}\rightarrow \overline{X}$. Define $S_{k_j} := U_{k_j}\Sigma_{k_j}V_{k_j}^{\top}$. If $k_j\in\mathcal{A}_1$, the optimal condition of $X_{k_j}$ leads to:
%		\begin{eqnarray*}
%			0 \in \partial g(X_{k_j}) + \frac{1}{\alpha}(X_{k_j} - S_{k_j})
%		\end{eqnarray*}
%		which implies:
%		\begin{eqnarray*}
%			\nabla h(X_{k_j}) - \nabla h(X_{k_j-1}) + \frac{1}{\alpha}(S_{k_j} - X_{k_j} + \alpha \nabla h(X_{k_j - 1})) + \frac{1}{\alpha}( X_{k_j-1} - X_{k_j}) \in \partial g(X_{k_j}) + \nabla h(X_{k_j}).
%		\end{eqnarray*}
%		Therefore:
%		\begin{eqnarray*}
%			\text{dist}(X_{k_j},\mathcal{X}^*)&&\leq (L + \frac{1}{\alpha})||X_{k_j} - X_{k_j-1}|| + \frac{\epsilon_{k_j}}{\alpha}.
%		\end{eqnarray*}
%		Similarly, if $k_j\in\mathcal{A}_2$, we have:
%		\begin{eqnarray*}
%			\text{dist}(X_{k_j},\mathcal{X}^*)&&\leq (L + \frac{1}{\alpha})||X_{k_j} - X_{k_j-1}||\\
%			&&\leq (L + \frac{1}{\alpha})||X_{k_j} - X_{k_j-1}|| + \frac{\epsilon_{k_j}}{\alpha}.
%		\end{eqnarray*}
%		Note that $\lim_{j\rightarrow\infty}||X_{k_j} - X_{k_j-1}|| = 0$ and $\lim_{j\rightarrow\infty}\epsilon_{k_j} = 0$, then the assertion follows directly.
%	\end{proof}
%	Note that $\overline{X}$ is said to be a stationary point of (\ref{matrix_problem}) if it holds that:
The following establishes the sufficient descent property in one iteration.

\begin{theorem}
\label{lem:descent}
Suppose  Assumption~\ref{Ass1} holds.  Let the sequences of
$\{\mathbf{x}^{k}\}_{k\ge 0}$ and $\{\mathbf{y}^k\}_{k\ge 0}$ be generated from Algorithm~\ref{alg:apgcn}.
Let $\h x_0=\h y_0$.
If $0<\alpha<1/L$ and $\sigma_{k}$ satisfies \eqref{sigmak}, then there exists a
constant $\widetilde C > 0$ such that the following 
property holds for $k\ge1$:
\begin{equation}\label{ineq:x_sufficient_descent}
F(\mathbf{x}^{k})
\;\leq\;
F(\mathbf{x}^{k-1})
- \widetilde C\|\mathbf{x}^{k}-\mathbf{y}^{k-1}\|^2
- \widetilde C\|\mathbf{x}^{k-1}-\mathbf{y}^{k-1}\|^2.
\end{equation}
Consequently,
\begin{equation}\label{adjent}
\lim_{k\to\infty}\|\mathbf{x}^{k}-\mathbf{x}^{k-1}\|
=
\lim_{k\to\infty}\|\mathbf{x}^{k}-\mathbf{y}^{k-1}\|
=
\lim_{k\to\infty}\|\mathbf{x}^{k-1}-\mathbf{y}^{k-1}\|
=0.
\end{equation}
Furthermore, there exists a constant $\widehat{F}$ such that
$
\lim_{k\to\infty} F(\mathbf{x}^{k}) = \widehat{F}.
$
\end{theorem}

\begin{proof}
From the update rule of the proximal-gradient step, we have
\begin{equation}\label{ieq:PG_descent}
F(\mathbf{x}^{k}) \le F(\mathbf{y}^{k-1})
- \frac{1}{2}\left(\frac{1}{\alpha}-L\right)\|\mathbf{x}^{k}-\mathbf{y}^{k-1}\|^2 .
\end{equation}
Next, we divide two cases to verify: (a) $\mathcal{S}^{k-1} \neq \emptyset$; and
(b) $\mathcal{S}^{k-1} = \emptyset$ for  $k\ge1$. \\
Case(a). It leads to $
F(\mathbf{y}^{k-1})
\le
F(\mathbf{x}^{k-1})
-\frac{\eta}{4}\|\mathbf{y}^{k-1}-\mathbf{x}^{k-1}\|^{2}
-\frac{c}{6}\|\mathbf{y}^{k-1}-\mathbf{x}^{k-1}\|^{3}.
$
Combining this inequality with  (\ref{ieq:PG_descent}), we obtain
\begin{equation*}
	\begin{aligned}
		F(\mathbf{x}^{k})
		&\le F(\mathbf{x}^{k-1})
		- \frac{1}{2}\left(\frac{1}{\alpha}-L\right)\|\mathbf{x}^{k}-\mathbf{y}^{k-1}\|^2
		- \frac{c}{6}\|\mathbf{x}^{k-1}-\mathbf{y}^{k-1}\|^3
		- \frac{\eta}{4}\|\mathbf{x}^{k-1}-\mathbf{y}^{k-1}\|^2 \notag \\
		&\le F(\mathbf{x}^{k-1})
		- \frac{1}{2}\left(\frac{1}{\alpha}-L\right)\|\mathbf{x}^{k}-\mathbf{y}^{k-1}\|^2
		- \frac{\eta}{4}\|\mathbf{x}^{k-1}-\mathbf{y}^{k-1}\|^2 \notag
	\end{aligned}
\end{equation*}
\noindent Let  $\widetilde C := \min\left\{\frac{1}{2}\left(\frac{1}{\alpha}-L\right), \frac{\eta}{4}\right\}$, then
(\ref{ineq:x_sufficient_descent}) follows directly.
Case (b). It leads to $\y^{k-1} = \x^{k-1}$.  (\ref{ineq:x_sufficient_descent}) also holds.
Rearranging (\ref{ineq:x_sufficient_descent}) and summing the resulting inequalities telescopically yields
$
\widetilde C \sum_{k=1}^{\infty}
\left(
\|\mathbf{x}^{k}-\mathbf{y}^{k-1}\|^2
+ \|\mathbf{x}^{k-1}-\mathbf{y}^{k-1}\|^2
\right)
\le F(\mathbf{x}^0) - \underline{F} < \infty .
$
Hence, \eqref{adjent} holds.
Finally, since $\{F(\mathbf{x}^{k})\}$ is nonincreasing and bounded below, there exists a constant $\widehat{F}$ such that
$\lim_{k\to\infty} F(\mathbf{x}^{k}) = \widehat{F}.$\qed
\end{proof}

\subsection{Global Convergence}
To establish the sequential convergence of
Algorithm~\ref{alg:apgcn}, we further impose the following assumption.

\begin{assumption}
\label{Ass2}
The objective function $F$ in (\ref{Problem}) is coercive, i.e., $F(\h x)\to +\infty$ as $\|\h x\|\to +\infty$.
\end{assumption}
As a consequence of Assumption \ref{Ass2}, all iterates remain in a compact set.
\begin{lemma}
Suppose that Assumptions~\ref{Ass1} and~\ref{Ass2} hold. Let $\{\mathbf{x}^{k}\}_{k\ge0}$ and $\{\mathbf{y}^{k}\}_{k\ge0}$ be generated from Algorithm~\ref{alg:apgcn}. Let $\h x_0=\h y_0$.
Denote $\overline{\mathcal{X}}$ as the set of all accumulation points of $\{\mathbf{x}^{k}\}$.
If $0<\alpha<1/L$
and $\sigma_{k}$ satisfies \eqref{sigmak},
then the following assertions hold:
\begin{itemize}
\item[(i)] The sequence $\{\mathbf{y}^k\}$ is bounded and admits the same  accumulation points set $\overline{\mathcal{X}}$.
\item[(ii)] $F(\overline{\mathbf{x}})= \widehat{F}$ for all $\overline{\mathbf{x}}\in\overline{\mathcal{X}}$,
where $\widehat{F}$ is defined in Theorem~\ref{lem:descent}.
\item[(iii)] Every accumulation point $\overline{\mathbf{x}}$ of $\{\mathbf{x}^k\}$ and $\{\mathbf{y}^k\}$ is a
critical point of problem~\eqref{Problem}, i.e.,
\begin{equation}\label{critical}
{{\bf 0}} \in\partial g(\overline{\mathbf{x}}) + \nabla h(\overline{\mathbf{x}}).\end{equation}
\end{itemize}
\end{lemma}
\begin{proof}
The proof is similar to \citep[Theorem 6.2]{boct2025full}, thus omitted here.\qed
\end{proof}

\begin{lemma}
\label{thm:subgradient_bound}
Suppose that Assumptions~\ref{Ass1} and~\ref{Ass2} hold. 
Let $\{\mathbf{x}^{k}\}$ and $\{\mathbf{y}^{k}\}$ be generated from Algorithm~\ref{alg:apgcn}.
If $0<\alpha<1/L$ and $\sigma_{k}$ satisfies \eqref{sigmak}, then the following  holds:
\begin{eqnarray}\label{ineq:dist_upper}
&\mathrm{dist}\bigl({{\bf 0}},\partial F(\mathbf{x}^{k})\bigr)
\le \left(\frac{1}{\alpha}+L\right)
\|\mathbf{x}^k - \mathbf{y}^{k-1}\| \nn\\
&\le \left(\frac{1}{\alpha}+L\right)
\bigl(\|\mathbf{x}^{k} - \mathbf{y}^{k-1}\|
+ \|\mathbf{x}^{k-1} - \mathbf{y}^{k-1}\|\bigr).
\end{eqnarray}
\end{lemma}

\begin{proof}
	 The proof is similar to \citep[Lemma 5.7]{Tao20}, thus omitted here.\qed
%Invoking the optimality condition of (\ref{PGstep}), we have
%$-\frac{1}{\alpha}(\mathbf{x}^{k} - \mathbf{y}^{k-1}) + \nabla f(\mathbf{x}^{k}) - \nabla f(\mathbf{y}^{k-1})\in \partial g(\mathbf{x}^{k}) + \nabla h(\mathbf{x}^{k}).$
%It implies that (\ref{ineq:dist_upper}) holds.\qed
\end{proof}
Next we establish the global convergence of  Algorithm~\ref{alg:apgcn} under the KL property of $F$.
The proof is similar to \citep[Theorem 6.11]{boct2025full}. 

\begin{theorem}
\label{thm:global_convegence}
Suppose that Assumptions~\ref{Ass1} and~\ref{Ass2} hold, and let $\{\mathbf{x}^{k}\}_{k\ge 0}$ and $\{\mathbf{y}^{k}\}_{k\ge 0}$ be generated from Algorithm~\ref{alg:apgcn}. Let $\h x_0=\h y_0$. Assume that $F$ satisfies the Kurdyka--Łojasiewicz (KL) property at every point in $\dom(\partial F)$. If $0<\alpha<1/L$ and $\{\sigma_k\}$ satisfies \eqref{sigmak}, then
$$
\sum_{k=1}^{\infty} \|\mathbf{x}^{k} - \mathbf{x}^{k-1}\| \leq\sum_{k=1}^{\infty} \left(||{\mathbf{x}}^{k} - {\mathbf{y}}^{k-1} || + ||{\mathbf{x}}^{k-1} - {\mathbf{y}}^{k-1} ||\right) < \infty.$$
Consequently, the sequences  $\{\mathbf{x}^{k}\}_{k\ge 0}$ and $\{\mathbf{y}^{k}\}_{k\ge 0}$  both converge to the same stationary point $\overline{\mathbf{x}}$ satisfying \eqref{critical}.
\end{theorem}

\begin{proof}
	Define $r_k := F(\mathbf{x}^k) - \widehat{F}$.
Let $\phi$ be the  desingularization function in Definition \ref{def:KL}. It yields that
\begin{equation}\label{ineq:KL_used}
\phi'(r_{k-1})\,
\mathrm{dist}({{\bf 0}},\partial F(\mathbf{x}^{k-1}))
\ge 1.
\end{equation}
Using the concave property of $\phi$ and let $k\ge2$,
\begin{equation*}
	\begin{aligned}
		&\phi(r_{k-1})-\phi(r_k) \ge \phi'(r_{k-1}) (r_{k-1}-r_k )\\
& \ge \widetilde{C}\phi'(r_{k-1})\left(\|\mathbf{x}^k - \mathbf{y}^{k-1}\|^2 +\|\mathbf{x}^{k-1} - \mathbf{y}^{k-1}\|^2\right) \\
		& \ge \frac{\|\mathbf{x}^k - \mathbf{y}^{k-1}\|^2 +\|\mathbf{x}^{k-1} - \mathbf{y}^{k-1}\|^2 }{\widehat M \left(\|\mathbf{x}^{k-1} - \mathbf{y}^{k-2}\|+ \|\mathbf{x}^{k-2} - \mathbf{y}^{k-2}\|\right)},
	\end{aligned}
\end{equation*}
where $\widehat M = (1/\alpha + L)/\widetilde{C}$. Rearrange the above, we have
\begin{equation*}
	\begin{aligned}
	&	\|\mathbf{x}^k - \mathbf{y}^{k-1}\| +\|\mathbf{x}^{k-1} - \mathbf{y}^{k-1}\|\\
& \leq \sqrt{2\widehat M(\phi(r_{k-1}) - \phi(r_{k}))\left(\|\mathbf{x}^{k-1} - \mathbf{y}^{k-2}\|+ \|\mathbf{x}^{k-2} - \mathbf{y}^{k-2}\|\right)}\\
		& \leq {\widehat M}(\phi(r_{k-1}) - \phi(r_{k})) + \frac{1}{2}\left(\|\mathbf{x}^{k-1} - \mathbf{y}^{k-2}\|+ \|\mathbf{x}^{k-2} - \mathbf{y}^{k-2}\|\right).
	\end{aligned}
\end{equation*}
So,
\begin{equation}
	\begin{aligned}\label{ineq:KL_telescope}
		&2(\|\mathbf{x}^k - \mathbf{y}^{k-1}\| +\|\mathbf{x}^{k-1} - \mathbf{y}^{k-1}\|)\\
& \leq {2}\widehat M(\phi(r_{k-1}) - \phi(r_{k})) + \left(\|\mathbf{x}^{k-1} - \mathbf{y}^{k-2}\| + \|\mathbf{x}^{k-2} - \mathbf{y}^{k-2}\| \right).
	\end{aligned}
\end{equation}
Combining (\ref{ineq:KL_telescope}) with $\phi > 0$, we have
%\begin{align*}
%		\sum_{k = K_0 + 2}^{K}\|\mathbf{x}^k - \mathbf{x}^{k-1}\| +\|\mathbf{x}^{k-1} - \mathbf{y}^{k-1}\| \leq 4M\phi(r_{K_0+1}) + \left(\|\mathbf{x}^{K_0+1} - \mathbf{x}^{K_0}\| + \|\mathbf{x}^{K_0} - \mathbf{y}^{K_0}\| \right).
%	\end{align*}
%	Let $K \to \infty$, we obtain 
$\sum_{k=1}^{\infty} \left(||{\mathbf{x}}^{k} - {\mathbf{y}}^{k-1} || + ||{\mathbf{x}}^{k-1} - {\mathbf{y}}^{k-1} ||\right) < \infty.$
Thus, $\{\mathbf{x}^k\}_{k\ge 0}$  converges to some $\overline{\mathbf{x}}$ satisfying (\ref{critical}). Moreover, $\{\y^k\}_{k\ge 0}$ also converges to $\overline{\mathbf{x}}$.
\qed
\end{proof}

%	\textcolor{red}{
%	Next, we review the standard error bound results under different KL exponents follows from the proof of \citep[Theorem 2]{attouch2009convergence}.
\begin{theorem}\label{thm:KL_rate}
Suppose that the assumptions of Theorem~\ref{thm:global_convegence} hold. 
Let $\{\x^k\}_{k\ge 0}$ and $\{\mathbf{y}^{k}\}_{k\ge 0}$ be the sequences generated from  Algorithm~\ref{alg:apgcn}. Let $\h x_0=\h y_0$.
Assume that $F$ satisfies the KL property with exponent $\theta \in (0,1)$ with desingularization function $\phi(s)=\mu^{-1} s^{1-\theta}$ with $\mu>0$. Then the following statements hold:
\begin{enumerate}
\item[(i)] For $\theta\in(0,1)$, there exists a constant $\widehat{C} > 0$ such that, for all sufficiently large $k$,
\begin{equation*}
	\|\x^k - \overline{\x}\|  \le \widehat C\left(\|\mathbf{x}^{k} - \mathbf{y}^{k-1}\| + \|\mathbf{x}^{k-1} - \mathbf{y}^{k-1}\| \right)^{\displaystyle{\min\left(1,\frac{1-\theta}{\theta}\right)}}.
\end{equation*}
\item[(ii)] If $\theta \in (0,\tfrac{1}{2}]$,  there exist constants $\widehat{C}_1 > 0$ and $\rho \in (0,1)$ such that, for all sufficiently large $k$,
\begin{equation*}
\|\x^k - \overline{\x}\| \le \widehat{C}_1 \rho^k.
\end{equation*}
\item[(iii)] If $\theta \in (\tfrac{1}{2},1)$,  there exists a constant $\widehat{C}_2 > 0$ such that, for all sufficiently large $k$,
\[
\|\x^k - \overline{\x}\| \le \widehat{C}_2 \, k^{-\frac{1-\theta}{2\theta-1}}.
\]
\end{enumerate}

\end{theorem}
%}
\begin{proof}
(i) Combining  (\ref{ineq:dist_upper}) and (\ref{ineq:KL_used}) with $k:=k+1$, and using $\phi'(r_k)=\mu^{-1}(1-\theta)r_k^{-\theta}$, we have
\begin{equation}
\left(\frac{1}{\alpha}+L\right)
\bigl(\|\mathbf{x}^{k} - \mathbf{y}^{k-1}\|
+ \|\mathbf{x}^{k-1} - \mathbf{y}^{k-1}\|\bigr)\ge  \mathrm{dist}(0,\partial F(\mathbf{x}^{k})) \ge \frac{\mu r_k^\theta}{1-\theta}.	\label{keyineq}
\end{equation}

\noindent It follows from (\ref{ineq:KL_telescope}) and define  $\Delta_k:=\sum_{p=k}^\infty \left(\|\mathbf{x}^{p+1} - \mathbf{y}^{p}\| +\|\mathbf{x}^{p} - \mathbf{y}^{p}\|\right)$, we have
\begin{equation*}
\begin{aligned}
&2\underbrace{\sum_{p=k+1}^{+\infty} \left(\|\mathbf{x}^p-{\mathbf{y}}^{p-1}\|+\|\mathbf{x}^{p-1}-\mathbf{y}^{p-1}\|\right)}_{\Delta_k}\nn\\
&\le {2}{\widehat M}\phi(r_k)+\underbrace{\sum_{p=k+1}^{+\infty}\left( 
\|\mathbf{x}^p-{\mathbf{y}}^{p-1}\|+\|\mathbf{x}^{p-1}-\mathbf{y}^{p-1}\|\right)}_{\Delta_k}+\|\mathbf{x}^k-\mathbf{y}^{k-1}\|+\|\mathbf{x}^{k-1}-\mathbf{y}^{k-1}\|\end{aligned}
\end{equation*}
\noindent By combining (\ref{keyineq}) with the above,  it yields that
\begin{equation}\label{keyin}
	\begin{aligned}		
		&\Delta_k\le {2}{\widehat M}\phi(r_k)+\|\mathbf{x}^k-\mathbf{y}^{k-1}\|+\|\mathbf{x}^{k-1}-\mathbf{y}^{k-1}\|\\
		&\le {2}\frac{\widehat M}{\mu}\left(\frac{1}{\mu}(1-\theta)(\frac{1}{\alpha}+L)[\|\mathbf{x}^{k} - \mathbf{y}^{k-1}\|
		+ \|\mathbf{x}^{k-1} - \mathbf{y}^{k-1}\|]\right)^{\frac{1-\theta}{\theta}}\\
		&+[\|\mathbf{x}^{k} - \mathbf{y}^{k-1}\|
		+ \|\mathbf{x}^{k-1} - \mathbf{y}^{k-1}\|].
	\end{aligned}
\end{equation}
Let 
$
\widehat C = \frac{{2}\widehat M}{\mu}\left( \mu^{-1}(1-\theta)\Big(\tfrac{1}{\alpha} + L\Big) \right)^{\frac{1-\theta}{\theta}} + 1.
$
 Thus, (i) is established.\\
%If $\theta \in (0,\tfrac{1}{2}]$, then $\frac{1-\theta}{\theta} \ge 1$. 
%Hence, the desired result follows directly by noting that 
%$\|\x^k - \overline{\x}\| < \Delta_k$. 
%If $\theta \in (\tfrac{1}{2},1)$, then $\frac{1-\theta}{\theta} \le 1$, 
%and the same conclusion still holds.
(ii)
For $\theta \in (0,\tfrac{1}{2}]$, it follows from (\ref{keyin}) that
$
\Delta_k \le \widehat C \, (\Delta_{k-1} - \Delta_k).
$
The remainder of the proof proceeds along the same lines as 
\citep[Theorem 2(ii)]{attouch2009convergence}.\\
\noindent
(iii)
From(\ref{keyin}), we obtain
$
\Delta_k^{\frac{\theta}{1-\theta}} 
\le \widehat C^{\frac{\theta}{1-\theta}} (\Delta_{k-1} - \Delta_k).
$
The desired result then follows by applying the argument of 
\citep[Theorem 2(i)]{attouch2009convergence}.
\qed
\end{proof}

%\begin{remark}From the proof of Theorem \ref{thm:KL_rate}, we see that for 
%%For $\theta\in(\tfrac12,1)$,
%%$$\|\x^k - \overline{\x}\|  \le \widehat C\left(\|\mathbf{x}^{k} - \mathbf{x}^{k-1}\| + \|\mathbf{x}^{k-1} - \mathbf{y}^{k-1}\| \right)^{\frac{1-\theta}{\theta}}.$$
%\end{remark}

\section{Finite Identification via Subspace Selection}
\label{sec:finite}
In this section, we specify a subspace selection strategy to identify the active manifold in finite time. %The following lemma establishes the relationship between $\mathcal{S}^k$ and the support set of the limit point.

\begin{lemma}
	\label{lem:identify_general}
	Suppose that  $\varepsilon_k$ defined in Algorithm \ref{alg:apgcn}  satisfying $\lim_{k\to+\infty}\varepsilon_k=0$. 
Let $\{\x^k\}_{k\ge 0}$ and $\{\mathbf{y}^{k}\}_{k\ge 0}$ be the sequences generated from  Algorithm~\ref{alg:apgcn}. Let $\h x_0=\h y_0$.
Assume that the assumptions in Theorem \ref{thm:global_convegence} hold. Denote $\overline{\x}$ be the limit point of $\{\x^k\}$, define $\overline{\mathcal{S}} = \supp(\overline{\x})$ and $\overline{n} = \|\overline{\mathbf{x}}\|_0$. Then,
	\begin{itemize}
		\item[(i)] there exists an index number $K_0$ such that
		$\overline{\mathcal{S}} \subseteq \mathcal{S}^{k},~ \forall k\ge K_{0};$
		\item[(ii)] for all $k\ge K_0$,
		$\overline n \le n_k\le \min\{{\lfloor \|\x_{k}-\overline \x\|^2}/{\varepsilon_k^2}\rfloor + \overline n,\|\x^k\|_0\}$.
	\end{itemize}
	%		\item[(iii)]
	%		Furthermore, assume that $F$ satisfies the KL property with exponent $\theta \in (0, 1)$, and one of the following conditions hold:
	%		\begin{itemize}
		%			\item[(iii-a)] Take $ \varepsilon_k = \widetilde{C} k^{-\gamma},~\gamma>0,~\widetilde{C}>0$ and $\overline{\x}$ be the limit point associated with $\{ \varepsilon_k\}$. $F$ satisfies KL property at $\overline{\x}$ with exponent $\theta\in(0,\frac{1}{2}]$.
		%			\item[(iii-b)] Take $ \varepsilon_k = \frac{\widetilde{C}}{\log(k+2)^{\gamma}},~\widetilde{C} >0,~\gamma>0,~\widetilde{C}>0$ and $\overline{\x}$ be the limit point associated with $\{ \varepsilon_k\}$. $F$ satisfies KL property at $\overline{\x}$ with exponent $\theta\in(0,1)$.
		%		\end{itemize}
	%		Then, there exists an index $K_1$ such that $\mathcal{S}_k =\overline{\mathcal{S}}, ~\forall k \ge K_1.$

	%		Additionally, if  $\theta \in (0, \tfrac{1}{2}]$, and let $\varepsilon_k = C \rho^{\tau_0 k}$, where $0 < \tau_0 < 1$.
	%		If $\theta \in (\tfrac{1}{2}, 1)$, and let $\varepsilon_k = C k^{-\tau_0 p}$, where $p := \frac{1 - \theta}{2\theta - 1} > 0$ and $0 < \tau_0 < 1$.
	%		Then, there exists an index $K_0$ such that
	%		$\mathcal{S}_k =\overline{\mathcal{S}}, ~\forall k \ge K_0.$
	%		In particular, under condition (i), one can take
	%		$
	%		K_0 = \left\lceil \frac{\log(2C/\rho_{\min})}{\tau_0 \log(1/\rho)} \right\rceil,$
	%		while for (ii), one can take
	%		$
	%		K_0 := \left\lceil \left(\frac{2C}{\rho_{\min}}\right)^{1 / (\tau_0 p)} - 1 \right\rceil.$
	%	\end{itemize}
	\end{lemma}
	
	\begin{proof}
%(i) The proof is similar to \citep[Theorem 4]{li2015global} and \citep[Theorem 6.11]{boct2025full}.\\
(i) If $\overline{\mathcal{S}} = \emptyset$, the result is obvious. Suppose $\overline{\mathcal{S}} \neq \emptyset$ and denote $\rho_{\min} := \min\{|\overline{x}_i|:i\in \overline{\mathcal{S}}\} > 0$.
Let $i\in\overline{\mathcal{S}}$, there exists an index $K_0$ such that $\varepsilon_k\le \rho_{\min}/2$ and $
\|\mathbf{x}^k - \overline{\mathbf{x}}\| \le \rho_{\min}/2$ when $k\ge K_0$.
Thus,
\begin{equation*}
	|x_i^k| \ge\;|\overline{x}_i| - |x_i^k - \overline{x}_i|\ge \rho_{\min} - \|\mathbf{x}^k - \overline{\mathbf{x}}\| \ge\frac{\rho_{\min}}{2}\ge \varepsilon_k,
\end{equation*}
which implies $i\in\mathcal{S}^k$. 
(ii) Let $k\geq K_0$, we may assume $\varepsilon_k >0$, otherwise the inequality holds directly.
Let  $T_{k}:=\mathcal{S}^{k}\setminus\overline{\mathcal S}$.
Clearly,
$\|{\mathbf{x}}^k-\overline {\mathbf{x}}\|^2\ge \sum_{i\in T_{k}} |x_i^k|^2 \ge \sharp(T_{k})\varepsilon^2_k=(n_k-\overline n)\varepsilon_k^2. $
The assertion follows directly. \qed

%	We consider the two cases separately.
%	{Case a:} If $\theta \in (0, \tfrac{1}{2}]$ and $\varepsilon_k = C \rho^{\tau_0 k}$, where $0 < \tau_0 < 1$, invoking Lemma \ref{lemma:SinM_general}, when $k \ge K_0$, we have $\overline{\mathcal{S}} \subseteq \mathcal{S}_k$. Combining this with Theorem \ref{thm:KL_rate}, it follows that $K_0$ is given by
%	$
%	K_0 = \left\lceil \frac{\log(2C/\rho_{\min})}{\tau_0 \log(1/\rho)} \right\rceil.
%	$
%	
%	To prove the inverse inclusion, we need to show that $n_k = \overline n$ for sufficiently large $k$. By Theorem \ref{thm:KL_rate}, we obtain:
%	\[
%	0 \le n_k - \overline n \le \left\lfloor \frac{\|\x^k - \overline \x\|^2}{\varepsilon_k^2} \right\rfloor \le \left\lfloor \rho^{2(1 - \tau_0) k} \right\rfloor = 0, ~ \forall k \ge K_0.
%	\]
%	Therefore, $\mathcal{S}_k \equiv \overline{\mathcal{S}}, ~ \forall k \ge K_0$.
%	{Case b:} If $\theta \in (\tfrac{1}{2}, 1)$ and $\varepsilon_k = C k^{-\tau_0 p}$, where $p := \frac{1 - \theta}{2\theta - 1} > 0$ and $0 < \tau_0 < 1$, we similarly obtain that
%	$
%	K_0 = \left\lceil \left( \frac{2C}{\rho_{\min}} \right)^{1 / (\tau_0 p)} - 1 \right\rceil.$
%	
%	And it holds that:
%	$
%	0 \le n_k - \overline n \le \left\lfloor \frac{\|\x^k - \overline \x\|^2}{\varepsilon_k^2} \right\rfloor \le \left\lfloor k^{-2p(1 - \tau_0)} \right\rfloor = 0, ~ \forall k \ge K_0.
%	$
%	Thus, $\mathcal{S}_k \equiv \overline{\mathcal{S}}, ~ \forall k \ge K_0$.
%	\(\square\)\qed
\end{proof}

%We first consider the case where $F$ satisfies KL property with exponent $\theta\in(0,1)$, see Definition \ref{def:KL_exponent}.

\begin{theorem}
\label{thm:Finite_identify_KL_general}
Suppose that all the assumptions of Theorem~\ref{thm:global_convegence} hold, and that $F$ satisfies the KL property at every point in $\dom(\partial F)$ with exponent $\theta \in (0,1)$.
Let $\{\x^k\}_{k\ge 0}$ and $\{\mathbf{y}^{k}\}_{k\ge 0}$ be the sequences generated from  Algorithm~\ref{alg:apgcn}, with $\h x_0=\h y_0$. Let $\overline{\x}$ denote its limit point with $\overline{\mathcal{S}} := \supp(\overline{\x})$. Take
$
\varepsilon_k = {\widehat{C}_3} \big(\|\mathbf{x}^{k} - \mathbf{y}^{k-1}\| + \|\mathbf{x}^{k-1} - \mathbf{y}^{k-1}\|\big)^{\gamma},
$
where ${\widehat{C}_3}> 0$ and $\gamma \in \big(0, \min(1,(1-\theta)/\theta)\big)$.
Then, for all sufficiently large $k$, it holds that
$
\mathcal{S}^k = \overline{\mathcal{S}}.
$
\end{theorem}
\begin{proof}
	Since $\lim_{k\to\infty}\varepsilon_k = 0$, it follows from Lemma \ref{lem:identify_general} that for sufficiently large $k$, it holds $\overline{\mathcal{S}} \subseteq \mathcal{S}^{k}$.
Next,  by invoking Lemma \ref{lem:identify_general} and Theorem \ref{thm:KL_rate}, we obtain for sufficiently large $k$:
\begin{equation*}
	0\le n_k - \overline n \le \left\lfloor \frac{\|\x^k - \overline \x\|^2}{\varepsilon_k^2} \right\rfloor \nn \le \left\lfloor \frac{\widehat C^2}{{\widehat{C}_3}^2}\left(\|\mathbf{x}^{k} - \mathbf{y}^{k-1}\| + \|\mathbf{x}^{k-1} - \mathbf{y}^{k-1}\|\right)^{2\min\{\frac{1-\theta}{\theta},1\}-2\gamma} \right\rfloor.
\end{equation*}
Thus, $n_k = \overline{n}$ for sufficiently large $k$. \qed
\end{proof}

\noindent If $F$ satisfies the KL property at each point in $\dom(\partial F)$ with exponent $\theta \!\in\! (0,1)$ while  $\theta$ is unknown, 
we may take
$
\varepsilon_k = \mathcal{O}\left( \big(\|\mathbf{x}^{k} - \mathbf{y}^{k-1}\| + \|\mathbf{x}^{k-1} - \mathbf{y}^{k-1}\|\big)^{\gamma_k} \right),
$
where
$
\gamma^k = o\!\left(\frac{1}{\left|\ln\big(\|\mathbf{x}^{k} - \mathbf{y}^{k-1}\| + \|\mathbf{x}^{k-1} - \mathbf{y}^{k-1}\|\big)\right|}\right)
$ at iteration $k$.

\section{Refined Local Convergence}\label{sec:local}
%	Then line 3 in Algorithm \ref{alg:apgcn} equals to:
%	\begin{eqnarray*}
	%		\overline{\mathbf{d}}^{k+1}\in{\arg\min}\left\{\langle \overline{\nabla} F(\mathbf{x}^{k+1}),\mathbf{d}\rangle +\frac{1}{2}\langle (\overline{\nabla}^2 F(\mathbf{x}^{k+1})+\eta I)\mathbf{d},\mathbf{d} \rangle+ \frac{\sigma_{k+1}}{6}||\mathbf{d}||^3\right\}
	%	\end{eqnarray*}
In this section, we further analyze the local convergence.
We assume that $F$ satisfies the conditions of Theorem~\ref{thm:Finite_identify_KL_general}.
Then there exists an index $K_1(\ge K_0)$ such that $\mathcal{S}^k = \overline{\mathcal{S}}$ for all $k \ge K_1$.
Consequently,  the iterates evolve on the identified manifold $\overline{\mathcal{S}}$, and
\begin{equation*}
	\mathbf{g}^k = \nabla_{\overline{\mathcal{S}}} F(\mathbf{x}^k), 
	\qquad
	H^k = \nabla^2_{\overline{\mathcal{S}}} F(\mathbf{x}^k).
\end{equation*}
%	and define the vector $\overline{\mathbf{d}}^k\in\mathbb{R}^n$ as:
%	\begin{equation*}
	%		\overline{\mathbf{d}}^k_{\mathcal{S}} = {\mathbf{d}}_{\mathcal{S}}^k,\ \overline{\mathbf{d}}^k_{\mathcal{S}^{c}} = 0.
	%	\end{equation*}

%Our analysis is more intricate than that of the classical cubic Newton method in the full space. This is due to the fact that, although \( \mathcal{S}^k \equiv \overline{\mathcal{S}} \) for sufficiently large \( k \), the support \( \supp(\mathbf{x}^k) \) does not necessarily coincide with \( \mathcal{S}^k \). Furthermore, the locally Hessian Lipschitz continuity of \( F \) is only valid for points that share the same support. To address this issue, we impose the following additional assumption:

%Our analysis is more delicate than that of the classical cubic Newton method in the full space. This is because the local Lipschitz continuity of the Hessian of $F$ holds locally with the points sharing the same support.
To proceed, we impose the following additional assumption.

\begin{assumption}\label{ass:partial_Hessian_continuous}
	For each ${\mathbf{x}} \in \dom \partial F$, let $\mathcal{S} = \supp(\x)$.
	\begin{enumerate}
		\item[(i)] There exists $\delta_{\mathbf{x}} > 0$ such that the restricted Hessian $\nabla^2_{{\mathcal{S}}} F(\mathbf{z})$ is continuous for all $\mathbf{z} \in \cB({\mathbf{x}}, \delta_{\mathbf{x}})$.\footnote{\noindent Indeed, Assumption~\ref{Ass1}(i) does not imply Assumption~\ref{ass:partial_Hessian_continuous}(i). 
			Consider the function $F: \mathbb{R}^2 \to \mathbb{R}$ defined by $F(\mathbf{x}) = g(\mathbf{x}) + h(\mathbf{x})$, where
			\[
			g(\mathbf{x}) = \frac{1}{6}|x_1|^3 \phi(x_2), 
			\quad 
			h(\mathbf{x}) = \frac{1}{2}(x_1^2 + x_2^2),
			\]
			and $\phi(t) = 1$ if $t > 0$ and $\phi(t) = 0$ if $t \leq 0$.
			%Clearly, for each $\mathbf{x}$ with $x_2\le 0$
			Then we have
			\[
			\mathrm{lip}_{\cal M}\,\nabla_{\mathcal S}^2 F(\x) =
			\begin{cases}
				0, & \text{if } \mathcal{S} = \{1\}, \\
				0, & \text{if } \mathcal{S} = \{2\}, \\
				0, & \text{if } \mathcal{S} = \{1,2\} \text{ and } x_2 < 0, \\
				1, & \text{if } \mathcal{S} = \{1,2\} \text{ and } x_2 > 0.
			\end{cases}
			\]
			Thus, $F$ satisfies Assumption~\ref{Ass1}(i) at each point $\mathbf{x}$.
			%and $\underline{F} = 0$.
			For the point of $\overline{\mathbf{x}} = (1,0)^{\top}$ with $\overline{\mathcal{S}} = \{1\}$, and
			$
			\nabla^2_{[1]} F(x_1,x_2) = x_1 \phi(x_2) + 1,$
			which is {\it not continuous} in any neighborhood of $\overline{\mathbf{x}}$. %Therefore, $F$ does not satisfy Assumption~\ref{ass:partial_Hessian_continuous}.
		}
		\item[(ii)] The Lipschitz modulus $\mathrm{lip}\nabla^2_{\mathcal{S}}F$ is lower semi-continuous at $\mathbf{x}$.
	\end{enumerate}
\end{assumption}

%\begin{remark}
%\textcolor{red}{(i) When $\delta$ is sufficiently small, it holds $\supp(\overline{\x}) \subseteq \supp(\x)$. Consequently, the restricted Hessian $\nabla_{\overline{\mathcal{S}}}^2 F$ is well defined on  $\cB(\overline{\x},\delta)$ when suitably choose $\delta$. (ii) When $g$ is separable, this assumption always holds.}
%\end{remark}

The following {Theorem} demonstrates that the exact cubic Newton step will be adopted, i.e. $\beta^k\equiv 1$ when
$k$ is  sufficiently large.
\begin{theorem}
	\label{lemma:beta_eq1}
	Let $\{\x^k\}_{k\ge 0}$ and $\{\mathbf{y}^{k}\}_{k\ge 0}$ be the sequences generated from  Algorithm~\ref{alg:apgcn}, with $\h x_0=\h y_0$. 
Assume that all assumptions in Theorem \ref{thm:global_convegence} hold. $\overline{\mathcal{X}}$ is the set of all accumulation points of $\{\mathbf{x}^{k}\}$, and assume that ${\bf 0}\not\in\overline{\mathcal{X}}$.
	Then,   $\x^k \to \overline{\mathbf{x}}$ and the following statements hold.
	\begin{itemize}
		\item[(i)] There exists ${\widehat{C}_5} > 0$ such that $\max_{k}\{\|\g^k\|,\|H^k\|,\|\mathbf{d}^k\|\}\le {\widehat{C}_5}$.\\
		Furthermore, suppose that Assumption \ref{ass:partial_Hessian_continuous} holds, then
		\item[(ii)] there exists a constant $\underline{\beta}>0$ such that $\beta_k\ge\underline{\beta}$ when $k$ is sufficiently large.
		\item[(iii)]  ${\mathbf{d}}^k\to 0$ as $k\to +\infty$.
		\item[(iv)]  there exists an index $\widehat K$ such that $\beta_k= 1$ when $k\ge\widehat K$.
		\item[(v)] It holds that
		\begin{equation}\label{key2}
			\|\nabla^2_{\overline{\mathcal{S}}}F(\x^k + t\mathbf{d}^k)- \nabla^2_{\overline{\mathcal{S}}}F(\x^k)\| \le \eta/2  + M(\overline{\x})t\|\mathbf{d}^k\|,~\forall t\in [0,1] 
		\end{equation}
		when $k\ge\widehat K$, where $M(\overline{\x}) = \frac{\underline{\omega} + 1}{2}\,\mathrm{lip}\,\nabla^2_{\overline{\mathcal{S}}}F(\overline{\x})$.
	\end{itemize}
\end{theorem}
\begin{proof}
	(i)
	Since $\lim_{k\to \infty}\varepsilon_k = 0$ and $\overline{\mathbf{x}}\neq {\bf 0}$, it must hold that $\mathcal{S}^k\neq\emptyset$ for sufficiently large $k$. Without loss of generality, we assume that $\mathcal{S}^k\neq \emptyset$ for all $k$.
	Assumption~\ref{Ass2} implies that there exists a constant ${\widehat{C}_4} > 0$ such that $\max_k\{\|\mathbf{g}^k\|,\|H^k\|\}\le {\widehat{C}_4}$. Define
	$
	\varphi_k(\mathbf{d}) = \langle \mathbf{g}^{k},\mathbf{d}\rangle + \frac{1}{2}\langle (H^{k}+\eta I)\mathbf{d}, \mathbf{d} \rangle + \frac{\sigma_{k}}{6}\|\mathbf{d}\|^3.
	$
	It is bounded below by
	$
	\underline{\varphi}(\mathbf{d}) = -{\widehat{C}_4} \|\mathbf{d}\| - \frac{{\widehat{C}_4}}{2} \|\mathbf{d}\|^2 + \frac{c}{9} \|\mathbf{d}\|^3.
	$
	So, $\{\mathbf{d}^k\}\subseteq \{\mathbf{d}:\underline{\varphi}(\mathbf{d})\le 0\}$. Consequently, $\{\mathbf{d}^k\}$ is bounded. There exists ${\widehat{C}_5} \,(\ge {\widehat{C}_4})$ such that $\max_{k}\{\|\mathbf{g}^k\|,\|H^k\|,\|\mathbf{d}^k\|\}\le {\widehat{C}_5}$.
	
	(ii) By Assumption~\ref{ass:partial_Hessian_continuous}, for $\overline{\x}$, there exists $\delta_{\overline \x}\in(0,\min_{i\in\overline{\mathcal{S}}} |\overline{x}_i|)$ such that $\nabla^2_{\overline{\mathcal{S}}}F(\x)$ is continuous on $\cB(\overline{\x},\delta_{\overline \x})$, and thus uniformly continuous. For  $M(\overline{\x}) = \frac{\underline{\omega} + 1}{2}\,\mathrm{lip}\,\nabla^2_{\overline{\mathcal{S}}}F(\overline{\x})$, 
	there exists $\delta_1>0$,  $\nabla^2_{\overline{\mathcal{S}}}F(\x)$ is Lipschitz continuous on $\cB(\overline{\x},\delta_1)\cap\overline{\mathcal{M}}$
	($\overline{\mathcal{M}}=\{\mathbf{x}:{\text{supp}}(\mathbf{x})=\text{supp}(\overline{\mathbf{x}})\}$) with modulus $M(\overline{\x})$.
	For the $\eta>0$ given in Algorithm~\ref{alg:apgcn}, there exists $0< \overline{\delta} \le \min(\delta_{\overline \x},\delta_1)$ such that
	\begin{equation*}
		\label{ineq:partial_Hessian_uniformly_continuous}
		\|\nabla^2_{\overline{\mathcal{S}}}F(\x) - \nabla^2_{\overline{\mathcal{S}}}F(\y)\| \le \eta/4,~\forall \x,\y\in\cB(\overline{\x},\overline{\delta}).
	\end{equation*}
	
	\noindent There exists $K_2$ such that $k\ge K_2$, $\x^k\in \cB(\overline{\x},\overline{\delta}/2)$.
	Define $\hat{\x}^k = {\text{Proj}}_{\overline{\mathcal{M}}}({\x}^k)$, where ${\text{Proj}}_{\overline{\mathcal{M}}}$ is the projection operator onto $\overline{\mathcal{M}}$.
	So, $\hat\x^k\in \cB(\overline{\x},\overline{\delta}/2)$ when $k\ge K_2$.
	Define $\widehat\beta = \min\{\overline{\delta}/(2{{\widehat{C}_5}}),1\}$. 
	We have that \begin{equation*}
		\label{suppxk}\supp(\x^k + t\beta \mathbf{d}^k) = \supp(\x^k)\end{equation*}
		 for $\beta\in(0,\widehat\beta]$ and  $t\in [0,1]$. 
	%Thus,   $\beta_k\ge \widehat\beta q$.
	Note that
	\begin{equation*}
		\label{ineq: inexact_Hessian_Lipschitz}
		\begin{aligned}
			&\|\nabla^2_{\overline{\mathcal{S}}}F(\x^k + t\beta\mathbf{d}^k)- \nabla^2_{\overline{\mathcal{S}}}F(\x^k)\|
			\le\underbrace{\|\nabla^2_{\overline{\mathcal{S}}}F({\x}^k + t\beta\mathbf{d}^k)- \nabla^2_{\overline{\mathcal{S}}}F(\hat{\x}^k + t\beta\mathbf{d}^k)\|}_{\textcircled{\small{1}}}  \nn\\
			&+\underbrace{ \|\nabla^2_{\overline{\mathcal{S}}}F(\hat{\x}^k + t\beta\mathbf{d}^k)- \nabla^2_{\overline{\mathcal{S}}}F(\hat{\x}^k)\|}_{\textcircled{\small{2}}}  + \underbrace{\|\nabla^2_{\overline{\mathcal{S}}}F(\hat{\x}^k)- \nabla^2_{\overline{\mathcal{S}}}F(\x^k)\|}_{\textcircled{\small{3}}}.
			%& \le \eta/4 + \eta/4 + |\nabla^2{\overline{\mathcal{S}}}F(\hat{\x}^k + t\beta\mathbf{d}^k)- \nabla^2_{\overline{\mathcal{S}}}F(\hat{\x}^k)| \
			%& \le \eta/2 + M(\overline{\x})|t\beta\mathbf{d}^k|.
		\end{aligned}
	\end{equation*}
	Since \begin{equation}\label{text1}{\x}^k + t\beta\mathbf{d}^k,\hat{\x}^k + t\beta\mathbf{d}^k,\hat{\x}^k,{\x}^k\in\cB(\overline{\x},{\overline\delta}),\; \hat{\x}^k + t\beta\mathbf{d}^k,\hat{\x}^k\in\cB(\overline{\x},\delta_1)\cap\overline{\mathcal{M}},\end{equation} then $\textcircled{\small{1}}\le \eta/4$,
	and $\textcircled{\small{3}}\le \eta/4.$
	Also, $\textcircled{\small{2}}\le M(\overline{\x})\|t\beta\mathbf{d}^k\|.$
	So,
	\begin{equation*}
		\|\nabla^2_{\overline{\mathcal{S}}}F(\x^k + t\beta\mathbf{d}^k)- \nabla^2_{\overline{\mathcal{S}}}F(\x^k)\|
		\le \eta/2 + M(\overline{\x})\|t\beta\mathbf{d}^k\|.
	\end{equation*}

	\noindent Using Lemma \ref{lem:inexact_cubic_upper} with $a = \eta/2$ and $M = M(\overline{\x})$ and $\mathbf{d} =\beta{\mathbf{d}}^k$, we have
	\begin{equation*}
		F(\mathbf{x}^k+\beta\mathbf{d}^k) \le F(\mathbf{x}^k) + \beta\langle \g^k, \mathbf{d}^k\rangle + \frac{\beta^2}{2} \langle H^k\mathbf{d}^k,\mathbf{d}^k\rangle + \frac{\eta\beta^2}{4}\|\mathbf{d}^k\|^2+ \frac{M(\overline{\x})\beta^3}{6}\|\mathbf{d}^k\|^3.
	\end{equation*}
	
	\noindent Moreover,  by Assumption \ref{ass:partial_Hessian_continuous}(ii),  $\mathrm{lip}\nabla^2_{\overline{\mathcal{S}}}F$ is lower semi-continuous at $\overline{\x}$, thus for any $\underline{\omega} >1$, we have $\mathrm{lip}\nabla^2_{\overline{\mathcal{S}}}F(\x^k)\geq \frac{1+\underline{\omega}}{2\underline{\omega}}\mathrm{lip}\nabla^2_{\overline{\mathcal{S}}}F(\overline{\x})$ when $k$ is sufficiently large. Therefore
	\begin{equation*}
		\begin{aligned}
			\sigma_{k}\geq \frac{2(\omega_k\mathrm{lip}\nabla^2_{{\mathcal{S}^k}}F(\x^k) + c)}{3} \overset{(a)}{\geq} \frac{2(\underline{\omega}\mathrm{lip}\nabla^2_{\overline{\mathcal{S}}}F(\x^k) + c)}{3}\overset{(b)}{\geq} \frac{2(M(\overline{\x})+c)}{3},
		\end{aligned}
	\end{equation*}
	where $(a)$ follows from $\mathcal{S}^k = \overline{\mathcal{S}}$ and $\omega_k\geq \underline{\omega}$, $(b)$ follows from the lower semi-continuity of $\mathrm{lip}\nabla^2_{\overline{\mathcal{S}}}F$ and the definition of $M(\overline{\x})$.
	
	Analogous to the proof of (\ref{key11}) in Theorem \ref{Newton_decrease},  we obtain that
	\begin{equation*}
		F(\mathbf{x}^k+\beta \mathbf{d}^k)\le F(\mathbf{x}^{k})-\frac{\eta\beta^2}{4}\big\|\mathbf{d}^{k}\big\|^2-\frac{c\beta^3}{6}\big\|\mathbf{d}^{k}\big\|^3,
	\end{equation*}
	thus $\beta$ is an acceptable stepsize.
	Therefore, by setting $\underline{\beta}=\widehat\beta q$, the assertion (ii) follows.\\
	(iii) 
	Since $||\mathbf{x}^k - \mathbf{y}^k||\to 0$, we have
	\begin{equation*}
		||\mathbf{d}^k||= \frac{||\mathbf{x}^k - \mathbf{y}^k||}{\beta_k}\le \frac{||\mathbf{x}^k - \mathbf{y}^k||}{\underline{\beta}}\to 0.
	\end{equation*}
	Thus, (iii) is valid.\\
	(iv) and (v). Since $\mathbf{d}^k\to 0$, (\ref{text1}) hold with $\beta_k = 1$ when $k\ge\widehat K$. Proceeding along the line of argument in (ii), the conclusion follows immediately. \qed
\end{proof}

\begin{corollary}\label{thm:approximate_second_order_stationary} 
Let $\{\x^k\}_{k\ge 0}$ and $\{\mathbf{y}^{k}\}_{k\ge 0}$ be the sequences generated from  Algorithm~\ref{alg:apgcn}, with $\h x_0=\h y_0$.  
Suppose that all the assumptions in Theorem \ref{thm:global_convegence} hold.  Moreover, Assumption \ref{ass:partial_Hessian_continuous} holds and $\sup_k \sigma_k := \overline{\sigma} < \infty$. Let $\eta>0$ be given in Algorithm \ref{alg:apgcn}.
	Then $\overline{\mathbf{x}}$ satisfies
	\begin{equation*}
		||\nabla_{\overline{\mathcal{S}}} F(\overline{\mathbf{x}})||= 0, \nabla^2_{\overline{\mathcal{S}}} F(\overline{\mathbf{x}})\succcurlyeq -\eta I.
	\end{equation*}
	%where $\overline{\mathbf{g}}=\nabla_{\overline{\mathcal{S}}} F(\overline{\mathbf{x}})$ and $\overline{H}=\nabla^2_{\overline{\mathcal{S}}} F(\overline{\mathbf{x}})$.
	
\end{corollary}
\begin{proof}
	%Invoking the optimal condition of $\mathbf{d}^{k}$, we have:
	%\begin{align}
	%& \g^{k} + H^{k}\mathbf{d}^{k} + \frac{\sigma_{k}}{2}||\mathbf{d}^{k}||\mathbf{d}^{k} + \eta \mathbf{d}^{k} = 0. \label{fopt}\\
	%&H^{k} \succcurlyeq -(\eta + \frac{{\sigma_k}}{2}||\mathbf{d}^{k}||)I.\label{sopt}
	%\end{align}
	%Therefore:
	From (\ref{gsub}) and (\ref{Hsub}),
	\begin{equation*}
		|| \g^{k}|| \le ({\widehat{C}_5}+\eta)||\mathbf{d}^{k}|| + \frac{\overline{\sigma}}{2}||\mathbf{d}^{k}||^2,~H^{k} \succcurlyeq -(\eta + \frac{\overline{\sigma}}{2}||\mathbf{d}^{k}||)I
	\end{equation*}
	where ${\widehat{C}_5}>0$ is defined in Theorem \ref{lemma:beta_eq1}.
	Let $k\to\infty$, we have $||\overline{\mathbf{g}}|| = 0$ and $\overline{H} \succcurlyeq -\eta I$.\qed
\end{proof}

%Next, we will study the local convergence behaviors of Algorithm \ref{alg:apgcn}.
%\begin{definition}
%	Let $\widehat{F}:\mathbb{R}^n\to\mathbb{R}$ be a twice differentiable function. Given $\mathbf{x}\in\mathbb{R}^n$ and $\varepsilon_g,\varepsilon_H>0$, we say that $\mathbf{x}$ is an $(\varepsilon_g,\varepsilon_H)$-approximate second-order stationary point of $\widehat{F}$ if
%	\[
%	\|\nabla \widehat{F}(\mathbf{x})\|\le \varepsilon_g
%	\quad\text{and}\quad
%	\lambda_{\min}\bigl(\nabla^2 \widehat{F}(\mathbf{x})\bigr)\ge -\sqrt{\varepsilon_H}.
%	\]
%\end{definition}
%
%
%\noindent The following theorem aims to show that the local complexity for the sequence $\{\y^k\}$ to reach an $(\varepsilon_g,\varepsilon_H)$-approximate second-order stationary point of the restricted function $F$ is $\mathcal{O}(\max\{\varepsilon_g^{-\frac{3}{2}},\varepsilon_H^{-\frac{3}{2}}\})$.
%To proceed,  we further introduce
%$
%	\mathbf{g}^k_{\mathbf{y}} = \nabla_{\overline{\mathcal{S}}} F(\mathbf{y}^k),\ H^k_{\mathbf{y}} = \nabla^2_{\overline{\mathcal{S}}} F(\mathbf{y}^k).
%$
\noindent The following theorem aims to show that the local complexity for the sequence $\{\y^k\}$ to reach an $(\varepsilon_g,\varepsilon_H)$-approximate second-order stationary point is $\mathcal{O}(\max\{\varepsilon_g^{-\frac{3}{2}},\varepsilon_H^{-\frac{3}{2}}\})$.
To proceed,  we further introduce
$
	\mathbf{g}^k_{\mathbf{y}} = \nabla_{\overline{\mathcal{S}}} F(\mathbf{y}^k),\ H^k_{\mathbf{y}} = \nabla^2_{\overline{\mathcal{S}}} F(\mathbf{y}^k).
$

\begin{theorem}\label{thm:local_complexity} Let $\{\x^k\}_{k\ge 0}$ and $\{\mathbf{y}^{k}\}_{k\ge 0}$ be the sequences generated from  Algorithm~\ref{alg:apgcn}, with $\h x_0=\h y_0$.  
	Suppose that all assumptions in Theorem~\ref{thm:global_convegence} hold.
	In addition, assume that Assumption~\ref{ass:partial_Hessian_continuous} holds and 
	$\sup_{k} \sigma_k := \overline{\sigma} < \infty$.
	Let $\varepsilon_g, \varepsilon_H > 0$ and choose
	\[
	0 < \eta \le \min\left\{
	\sqrt{\frac{(M(\overline{\x})+\overline{\sigma})\,\varepsilon_g}{9}},
	\;\;
	\frac{2\overline{\sigma}+2M(\overline{\x})}{6\overline{\sigma}+9M(\overline{\x})}
	\sqrt{\varepsilon_H}
	\right\},
	\]
	where $M(\overline{\x})$ is defined in Theorem \ref{lemma:beta_eq1}.
	%\[
	%M(\overline{\x}) := \frac{(\underline{\omega} + 1)}{2}
	%\,\mathrm{lip}_{\overline{\mathcal{S}}}\nabla^2 F(\overline{\x}).
	%\]
	Then, the iteration complexity required to obtain an $(\varepsilon_g,\varepsilon_H)$-approximate
	second-order stationary point is bounded by
	\[
	\widehat K+ \mathcal{O}\!\left(\max\left\{\varepsilon_g^{-\frac{3}{2}},\,
	\varepsilon_H^{-\frac{3}{2}}\right\}\right),
	\]
	where $\widehat K$ is defined in Theorem \ref{lemma:beta_eq1}.
\end{theorem}
\begin{proof}
	We first show that if $\|\mathbf{d}^{k}\|\le \frac{3\eta}{M(\overline{\x})+\overline{\sigma}}$, then $\mathbf{y}^{k}$ is an $(\varepsilon_g, \varepsilon_H)$-approximate second-order stationary point. To this end, invoking \eqref{gsub}, we have
	\begin{equation}
		\begin{aligned}
			\| \mathbf{g}_{\mathbf{y}}^{k}\|
			&= \left\|\mathbf{g}_{\mathbf{y}}^{k}-\mathbf{g}^{k} -  H^{k}\mathbf{d}^{k} - \frac{\sigma_{k}}{2}\|\mathbf{d}^{k}\|\mathbf{d}^{k} - \eta \mathbf{d}^{k}\right\| \\
			&\le \left\|\mathbf{g}_{\mathbf{y}}^{k}-\mathbf{g}^{k} -  H^{k}\mathbf{d}^{k} \right\| + \frac{\sigma_{k}}{2}\|\mathbf{d}^{k}\|^2+ \eta \|\mathbf{d}^{k}\| \\
			&\overset{(a)}{\le} \frac{M(\overline{\x})+\overline{\sigma}}{2}\|\mathbf{d}^{k}\|^2+ \frac{3\eta}{2} \|\mathbf{d}^{k}\|,
			\label{ieq:gy_bound}
		\end{aligned}
	\end{equation}
	where (a) follows from \eqref{ineq:grad_var_bound} and \eqref{key2}.  
	Consequently,
	$
	\| \mathbf{g}_{\mathbf{y}}^{k}\| \le \frac{9\eta^2}{M(\overline{\x})+\overline{\sigma}} \le \varepsilon_g.
	$
	
	Similarly, combining \eqref{Hsub} with \eqref{key2}, we obtain
	\begin{equation}
		H_{\mathbf{y}}^{k} \succeq H^{k} - M(\overline{\x})\|\mathbf{d}^{k}\| I - \frac{\eta}{2}I
		\succeq -\frac{3\eta}{2} I -\frac{ 2M(\overline{\x})+\overline{\sigma}}{2}\|\mathbf{d}^{k}\|I.
		\label{ieq:Hy_bound}
	\end{equation}
	Hence,
	$
	H_{\mathbf{y}}^{k} \succeq - \frac{\eta(6\overline{\sigma}+9M(\overline{\x}))}{2\overline{\sigma}+2M(\overline{\x})}I 
	\succeq -\sqrt{\varepsilon_H}I.
	$
	Therefore, $\mathbf{y}^{k}$ is an $(\varepsilon_g, \varepsilon_H)$-approximate second-order stationary point.
	 Next, define the index $T$ as the first iteration (after performing exact cubic Newton steps, i.e., $\beta_k=1$ for $k\ge\widehat K$) such that an $(\varepsilon_g, \varepsilon_H)$-approximate second-order stationary point is reached:
	\begin{equation*}
		T:=\inf_{k\ge\widehat K}\left\{k: \|\mathbf{g}^{k}_{\mathbf{y}}\| \le \varepsilon_g~\text{and}~ \lambda_{\min}(H^k_{\mathbf{y}}) \ge -\sqrt{\varepsilon_H} \right\}.
	\end{equation*}
	
	\noindent Since $\eta < \sqrt{\varepsilon_H}$, it follows from Corollary \ref{thm:approximate_second_order_stationary} that $T<\infty$. 
	For any $k\in[\widehat K, T-1]$, we must have $\|\mathbf{d}^{k}\|> \frac{3\eta}{M(\overline{\x})+\overline{\sigma}}$, i.e.,
	$
	\eta < \frac{M(\overline{\x})+\overline{\sigma}}{3}\|\mathbf{d}^{k}\|.$
	Invoking \eqref{ieq:gy_bound} and \eqref{ieq:Hy_bound}, we obtain
	\begin{equation*}
		\begin{aligned}
			\| \mathbf{g}_{\mathbf{y}}^{k}\| < (M(\overline{\x})+\overline{\sigma})\|\mathbf{d}^{k}\|^2, \;
			H_{\mathbf{y}}^{k} \succeq -\left(\frac{3M(\overline{\x})}{2}+\overline{\sigma}\right)\|\mathbf{d}^{k}\| I.
		\end{aligned}
	\end{equation*}
From the second, we have that 
\begin{eqnarray*}\lambda_{\min}(H_{\mathbf{y}}^{k})\ge -\left(\frac{3M(\overline{\x})}{2}+\overline{\sigma}\right)\|\mathbf{d}^{k}\|.\end{eqnarray*}
Therefore, 
\begin{eqnarray*}\|\h d^k\|\ge \frac{2[-\lambda_{\min}(H_{\mathbf{y}}^{k})]_+}{3M(\overline{\x})+2\overline{\sigma}}.\end{eqnarray*}
	Combining the above inequalities yields
	\begin{eqnarray*}
			\|\mathbf{d}^{k}\| 
			> \max\left\{\sqrt{\frac{\| \mathbf{g}_{\mathbf{y}}^{k}\|}{M(\overline{\x})+\overline{\sigma}}},\,
			\frac{2[-\lambda_{\min}(H_{\mathbf{y}}^{k})]_+}{3M(\overline{\x})+2\overline{\sigma}}\right\} >\min\left\{\sqrt{\frac{\varepsilon_g}{M(\overline{\x})+\overline{\sigma}}},\,
			\frac{2\sqrt{\varepsilon_H}}{3M(\overline{\x})+2\overline{\sigma}}\right\},		
	\end{eqnarray*}
	where the last is due to either $\|\mathbf{g}^{k}_{\mathbf{y}}\| > \varepsilon_g$ or $\lambda_{\min}(H^k_{\mathbf{y}}) < -\sqrt{\varepsilon_H}$
	when $k\in[\widehat K, T-1]$.
	Combining with the sufficient descent property, we have for $k\in[\widehat K, T-1]$,
	\begin{equation*}
			\begin{aligned}
			F({\mathbf{x}}^{k+1})
			&\le F(\mathbf{y}^{k}) \le F(\mathbf{x}^{k})-\frac{\eta}{4}\big\|\mathbf{x}^{k}-\mathbf{y}^{k}\big\|^2-\frac{c}{6}\big\|\mathbf{x}^{k}-\mathbf{y}^{k}\big\|^3 \\
			&\le F(\mathbf{x}^{k})-\frac{c}{6}\|\mathbf{d}^{k}\|^3 \\
			&<F(\mathbf{x}^{k}) - \frac{c}{6}\min\left\{\sqrt{\frac{\varepsilon_g}{M(\overline{\x})+\overline{\sigma}}},\frac{2\sqrt{\varepsilon_H}}{3M(\overline{\x})+2\overline{\sigma}}\right\}^3 \\
			&\le F(\mathbf{x}^{k}) - \frac{c}{6}\min\left\{\sqrt{\frac{1}{M(\overline{\x})+\overline{\sigma}}},\frac{2}{3M(\overline{\x})+2\overline{\sigma}}\right\}^3 \min\{\varepsilon_g^{3/2},\varepsilon_H^{3/2}\},
		\end{aligned}
	\end{equation*}
	where the first inequality follows from \eqref{ieq:PG_descent}.
	Rearranging the above inequality yields
	\begin{equation*}
		T-\widehat K
		\le \left\lceil \frac{6(F(\x^{\widehat K})-\underline{F})}{c}
		\min\left\{\sqrt{\frac{1}{M(\overline{\x})+\overline{\sigma}}},\frac{2}{3M(\overline{\x})+2\overline{\sigma}}\right\}^{-3}
		\min\{\varepsilon_g^{-3/2},\varepsilon_H^{-3/2}\} \right\rceil.
	\end{equation*}
	Thus, the assertion follows. \qed
\end{proof}

\section{Numerical Results}\label{sec:numerical}
%We presents some numerical experiments to demonstrate the efficiency of our algorithm.

%\subsection{A Practical Version}
In this section, inspired by \citep{cartis2011adaptive}, we propose an adaptive variant of the Subspace Cubic Newton-Proximal Gradient method, referred to as ASCN-PG, to avoid the need for estimating the cubic regularization parameter $\sigma_k$. The resulting scheme is summarized in Algorithm~\ref{Ag2}.

\begin{algorithm}[!htpb]
	\small
	\caption{Adaptive Subspace Cubic Newton-Proximal Gradient (ASCN-PG)}	
	\begin{algorithmic}[1]
		\Require{ $\mathbf{y}_0, \ 0<\alpha < 1/L,\ c>0,\ \eta>0,\ 0 < q < 1, {\omega_k\geq \underline{\omega}>1,}\  \varepsilon_k >0,\ 0<\eta_1<\eta_2<1, \ 0<\gamma_1<1<\gamma_2,\ 0<\sigma_{\min}<\sigma_{\max},\ \sigma_1 > 0$, \tt MaxIt.}
		\For {$k = 1:{\texttt{MaxIt}}$}
		
		\State \textbf{(Proximal--gradient step)}
		$\mathbf{x}^{k}\in
		\prox_{\alpha g}
		\bigl(\mathbf{y}^{k-1}-\alpha\nabla h(\mathbf{y}^{k-1})\bigr).$
		
		\State \textbf{($\varepsilon$-Active subspace identification)}
		$
		\mathcal S^{k}:=\{\,i:\ |x_i^{k}|\ge \varepsilon_k\,\}, \; n_k:=\sharp(\mathcal S^{k}).
		$
		\If{$\mathcal{S}^k \neq\emptyset$}
		
		\State \textbf{(Restricted derivatives)} $
		\g^{k}=\nabla_{\mathcal S^{k}}F(\mathbf{x}^{k}),\;
		H^{k}=\nabla^{2}_{\mathcal S^{k}}F(\mathbf{x}^{k}).$
		\State \textbf{(Subspace cubic--regularized Newton step)}
		$$\mathbf{d}^{k}\in\mathop{\arg\min}_{\mathbf{d}\in\R^{n_k}}
		\left\{
		\langle \g^{k},\mathbf{d}\rangle
		+\frac12\langle (H^{k}+\eta I)\mathbf{d},\mathbf{d}\rangle
		+\frac{\sigma_k}{6}\|\mathbf{d}\|^{3}
		\right\}.
		$$
		
		\State \textbf{(Backtracking line search)}
		Set $\beta_{k,0}=1$.
		\For{$j_k = 0,1,...$}
		\State Define $
		\mathbf{y}^{k,j_k}_{\mathcal S^{k}}=\mathbf{x}^{k}_{\mathcal S^{k}}+\beta_{k,j_k}\mathbf{d}^{k},
		\;
		\mathbf{y}^{k,j_k}_{(\mathcal S^{k})^{c}}=\mathbf{x}^{k}_{(\mathcal S^{k})^{c}}.$
		\If{$F(\mathbf{y}^{k,j_k})\le F(\mathbf{x}^{k})-\frac{\eta}{4}||\mathbf{x}^{k}-\mathbf{y}^{k,j_k}||^2-\frac{c}{6}||\mathbf{x}^{k}-\mathbf{y}^{k,j_k}||^3$}
		\State{$\mathbf{y}^{k} = \mathbf{y}^{k,j_k},\beta_k = \beta_{k,j_k}$}
		\State{break}
		\Else
		\State{$\beta_{k,j_k+1} = q*\beta_{k,j_k}$}
		\EndIf
		\EndFor
		
		\State{\textbf{(Adjust the regularized parameter)}}
		\State{Define $m_{k}(\mathbf{d}) = F(\x^k) +  \langle \g^{k},\ \mathbf{d}\rangle +\frac{1}{2}\langle (H^{k}+\eta I)\mathbf{d},\ \mathbf{d} \rangle+ \frac{\sigma_{k}}{6}||\mathbf{d}||^3$}
		\State{Let {$\rho_k = {(F(\x^k) - F(\y^k))}/(F(\x^k) - m_k(\beta_k\mathbf{d}^k))$}}
		\If{$\rho_k \ge \eta_2$}
		%\Comment{very successful}
		\State $\sigma_{k+1} = \max\{\sigma_{\min},\gamma_1\sigma_k\}$
		\ElsIf{$\rho_k \le \eta_1$}
		%\Comment{unsuccessful}
		\State $\sigma_{k+1} = \min\{\sigma_{\max}, \gamma_2\sigma_k\}$
		\Else
		%\Comment{successful}
		\State $\sigma_{k+1} = \sigma_k$
		\EndIf
		\Else
		\State{Set $\mathbf{y}^k = \mathbf{x}^k$.}
		\EndIf
		\If{Termination Criterion is True}
		\State{break}
		\EndIf
		\EndFor
		\State{{\bf Output} ${\mathbf{y}_{k}}$.}
	\end{algorithmic}\label{Ag2}

\end{algorithm}
To illustrate the efficiency of the newly-developed framework, we  compare SCN-PG and ASCN-PG with two other algorithms: Proximal Gradient method (PG) \citep{beck2017first} given by:
$$
\x^k\in\underset{\mathbf{x}}{\arg\min}\left\{g(\mathbf{x}) + \frac{1}{2\alpha}||\mathbf{x} - (\x^{k-1}- \alpha\nabla h(\x^{k-1}))||^2\right\},
$$
and the Manifold Acceleration truncated Newton method (ManAcc) \citep{bareilles2023newton}.
%To measure the performance of tested algorithms, we record the following measurements: objective function value (Obj), number of nonzero element (NNZ), the  dimension of subspace for Newton Cubic update ($n_k$), the CPU time (CPU), the Iteration numbers (Iter), the norm of restricted gradient to the current support set ($\|\g^k\|$), the minimal eigenvalue of the restricted Hessian to the current support set ($\lambda_{\min}(H^k)$)	and the
%Additionally, we also compute the successful ratio (SR) (See section \ref{sec:recovery}) and the classification accuracy (Acc) (See section \ref{sec:logistic})
%All these tested algorithms, we use the uniform  termination criterion:
For all tested algorithms, we use the following uniform termination criterion
\begin{equation}\label{stoppping}
	\frac{||\x^{k+1}-\x^k||}{\max\{||\x^{k}||,{\text{eps}}\}}<\texttt{RTol},~\text{or}~ k> \texttt{Maxit},~\text{or}~\text{CPUtime}> \texttt{CTol}~\text{seconds}.
\end{equation}
Define the KKT residual of (\ref{Problem}) by:
$
	\mathrm{KKT}(\x) = \inf_{\xi\in \partial g(\x)}\|\nabla h(\x) + \xi\|.
$
All reported results  are obtained by averaging over 30 independent trials.

\subsection{Sparse Signal Recovery}\label{sec:recovery}
In this section, we consider the sparse recovery model proposed in \citep{Tao20}:
\begin{equation}
	\min_{\x\in\mathbb{R}^n} 
	\lambda\frac{\|\x\|_1}{\|\x\|_2} 
	+ \frac{1}{2}\|A\x-\mathbf{b}\|_{2}^{2},
	\label{prob:L1oL2_recovery}
\end{equation}
where $A\in\mathbb{R}^{m\times n}$ with $m\ll n$, $\mathbf{b}\in\mathbb{R}^{m}$, and $\lambda>0$.
We consider two types of sensing matrices.
(i) Gaussian matrix (Gauss).
	The matrix $A$ is drawn from $\mathcal{N}(0,\Sigma)$ with covariance
	$
	\Sigma = (1-r)I_n + r\mathbf{1}\mathbf{1}^\top,
	\quad r\in(0,1).$
	(ii) Oversampled DCT matrix (ODCT).
	Let $A=[\mathbf{a}_1,\ldots,\mathbf{a}_n]$, where
	$
	\mathbf{a}_j = \frac{1}{\sqrt{m}} 
	\cos\!\left(2\pi \mathbf{w}_j F \right),
	\quad j=1,\ldots,n,$
	where $\mathbf{w}\in[0,1]^m$ is sampled uniformly and $F\in\mathbb{R}_+$ controls the coherence.
We set $m=64$, $n=1024$. The ground-truth signal $\x^*$ is $K$-sparse with
$K \in \{2,4,6,8,10\},$
and its nonzero entries are generated by
$\texttt{x(supp) = sign(randn(K,1)).*1e2.*randn(K,1)}.
$
The observation vector is given by $\mathbf{b}=A\x^*$.
We consider the following parameter configurations:
ODCT: $F \in \{10,15,20\}$,
Gaussian: $r \in \{0.1,0.2,0.3\}$.
We denote each instance by $(\mathrm{ODCT},K,F)$ and $(\mathrm{Gauss},K,r)$, respectively.

In \eqref{prob:L1oL2_recovery}, we set $\lambda=10^{-2}$ and let
$h(\x)=\frac{1}{2}\|A\x-\mathbf{b}\|_2^2,$
$g(\x)=\lambda\frac{\|\x\|_1}{\|\x\|_2}
$ for all these tested algorithms.
For PG, we set $\alpha=2$. For ManAcc,
we set $m_1=10^{-4}$, $\theta=0.5$.
To enhance stability, we use the modified Hessian $H^k+\eta I$ with $\eta=10^{-4}$ and $\alpha=2$.
For SCN-PG,
we set $\alpha=2$, $c=10^{-3}$, $\eta=10^{-10}$, $q=0.95$, and
$
M(\x^k)=\frac{256\sqrt{\|\x^k\|_0}}{\|\x^k\|^3}, 
\;
\sigma_k=M(\x^k)+c.$ threshold is chosen as
$
\varepsilon_k = 
\left(
\tau_1\|(\x^{k}-\y^{k-1})/\alpha\|
+\tau_2\|(\x^{k-1}-\y^{k-1})/\alpha\|
\right)^{\tau_3},$
with $\tau_1=0.99$, $\tau_2=0.01$, $\tau_3=0.85$.
For ASCN-PG,
we set $\alpha=2$, $c=10^{-3}$, $\eta=10^{-10}$, $q=0.95$, and
$\eta_1=0.1,\ \eta_2=0.9,\ 
\gamma_1=0.5,\ \gamma_2=2,
\sigma_{\min}=2c,\
\sigma_{\max}=100\sigma_{\min},\
\sigma_1=25\sigma_{\min}.
$
The threshold $\varepsilon_k$ is chosen as in SCN-PG.
We initialize all methods using the solution of the $L_1$-regularized problem
$\min_{\x\in\mathbb{R}^n} 
\lambda_0\|\x\|_1 
+ \frac{1}{2}\|A\x-\mathbf{b}\|_2^2,$
with $\lambda_0=10^{-3}$.
All these tested algorithms are terminate by (\ref{stoppping})  with $
\texttt{RTol}=10^{-8},\
\texttt{CTol}=50,$
and $\texttt{Maxit}=200000$.

\begin{figure}[t]
	\centering
	\begin{tabular}{cc}
		%\toprule
		%$(\text{ODCT},10,20)$ & $(\text{Gauss},10,0.3)$ \\
		%\midrule
		\includegraphics[scale=.4]{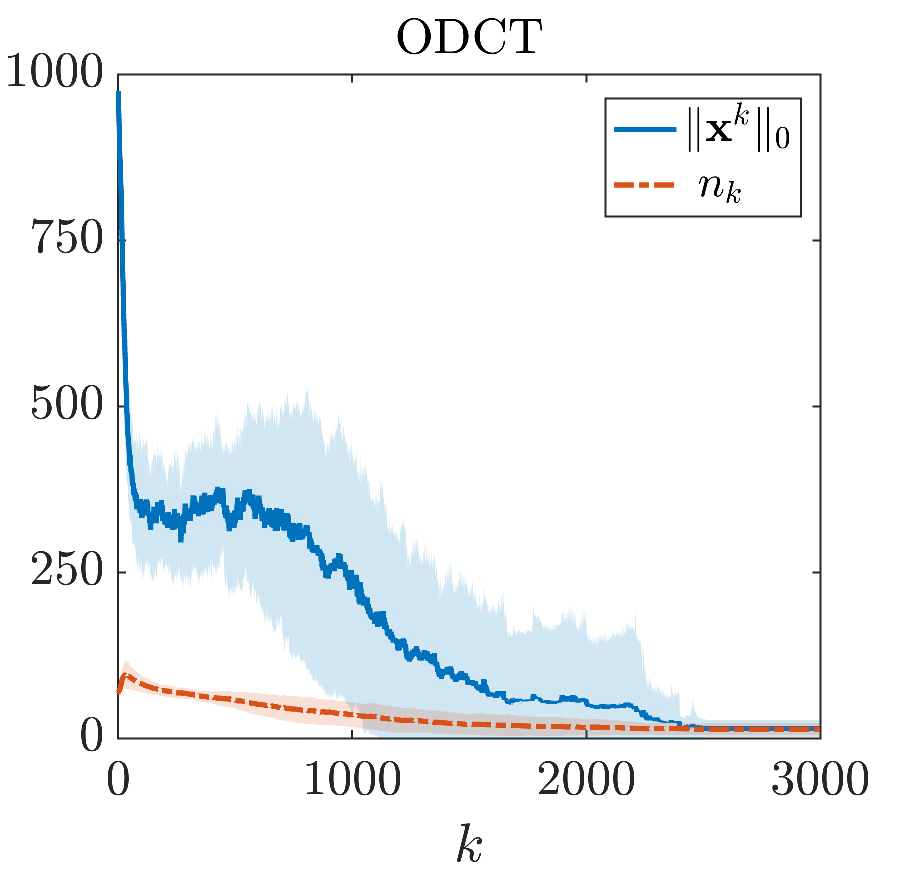} &
		\includegraphics[scale=.4]{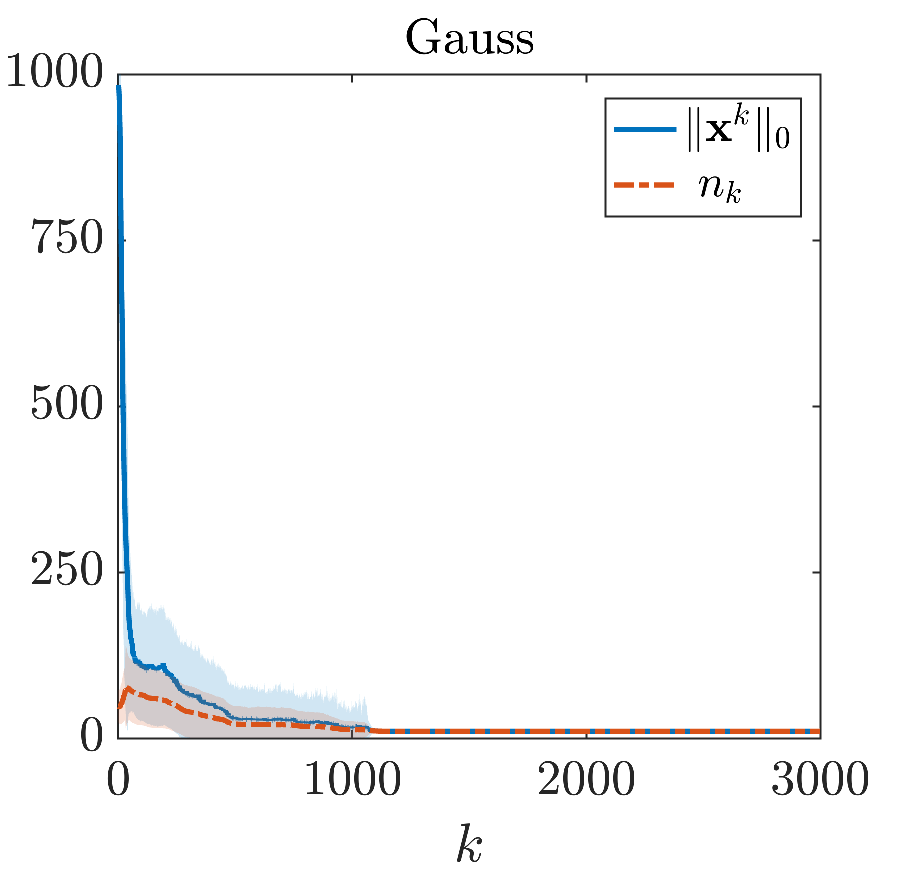} \\
		%\bottomrule
	\end{tabular}
	\caption{Histories of $\|\mathbf{x}^k\|_0$ and $n_k$ for SCN-PG under two scenarios:
		$(\mathrm{ODCT},10,20)$ (left) and $(\mathrm{Gauss},10,0.3)$ (right).}
	\label{fig:nnz_vs_nk}
\end{figure}
%Next, we study the relationship between $\|\x^k\|_0$ and $n_k$ for ASCN-PG.

Figure~\ref{fig:nnz_vs_nk} depicts the histories of $\|\x^k\|_0$ and $n_k$ with respect to the iteration number when applying SCN-PG to two scenarios: {$(\mathrm{ODCT},10,20)$ and $(\mathrm{Gauss},10,0.3)$. 
The curves correspond to averaged values over multiple runs, while the shaded regions indicate one standard deviation above and below the mean.
In the early stage, $\|\x^k\|_0$ is significantly larger than $n_k$, indicating that although the iterates have relatively large supports, the cubic Newton step is performed within a much lower-dimension subspace. 
After a finite number of iterations, the two quantities coincide, which is consistent with the finite identification established in our paper.
Moreover, the trajectory of $\|\x^k\|_0$ exhibits noticeable oscillations, whereas $n_k$ remains comparatively stable. This contrast highlights the robustness and effectiveness of the proposed subspace selection strategy.
The behavior of ASCN-PG is qualitatively similar and is therefore omitted for brevity.

\begin{figure}[t]
	\centering
	\begin{tabular}{cc}
		%\toprule
		%ODCT & Gauss \\
		%\midrule
		%\includegraphics[scale=.37]{APG_CN_figure/emperical_CDF_kkt_ODCT.eps} &
		%\includegraphics[scale=.37]{APG_CN_figure/emperical_CDF_kkt_Gauss.eps} \\
		\includegraphics[scale=.4]{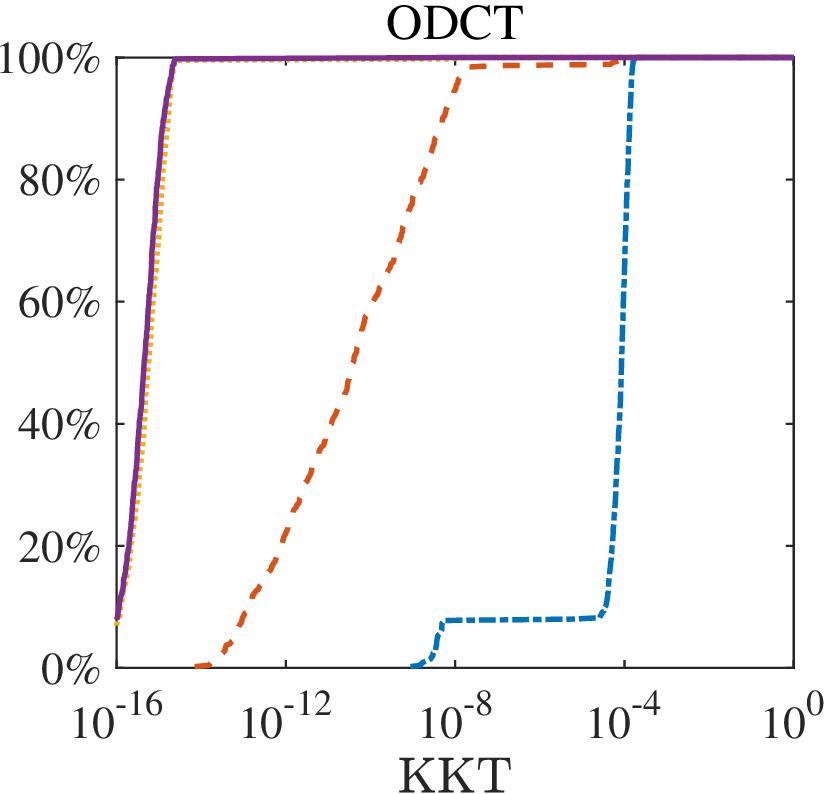} &
		\includegraphics[scale=.4]{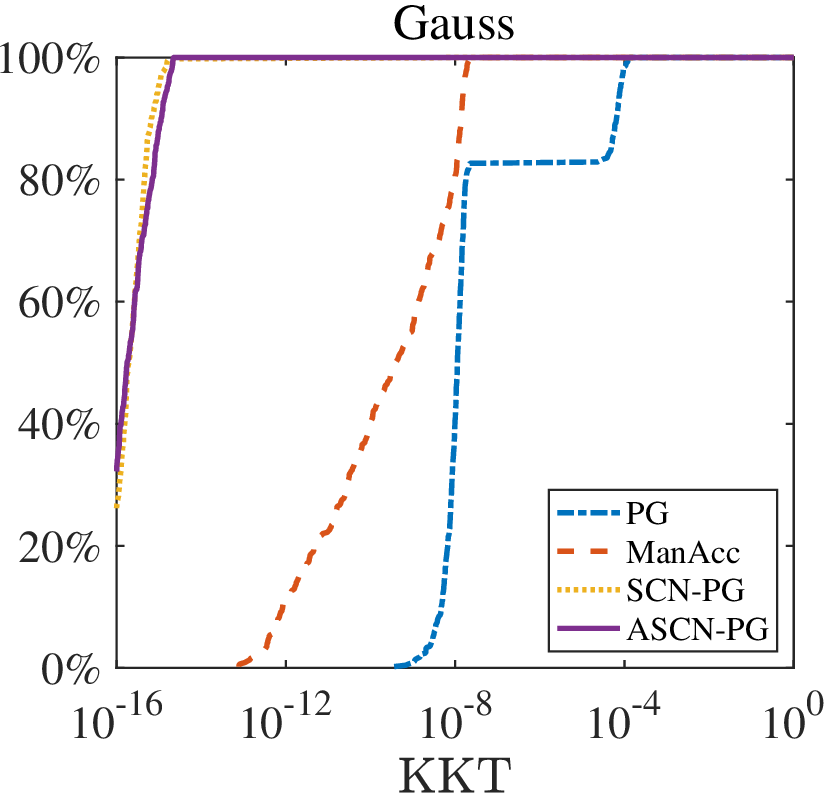} \\
		%\bottomrule
	\end{tabular}
	\caption{The empirical cumulative distribution of the KKT.}
	\label{fig:emperical_CDF_kkt}
\end{figure}

%	\begin{figure}[t]
	%		\centering
	%		\subfigure{%
		%			\includegraphics[scale = .4]{APG_CN_figure/perf_profile_cup_ODCT.eps}}%
	%		\hspace{0.1mm}%
	%		\subfigure{%
		%			\includegraphics[scale = .4]{APG_CN_figure/perf_profile_cup_Gauss.eps}}%
	%		\caption{The performance profile.}
	%		\label{fig:performance_profile}
	%	\end{figure}
\begin{figure}[t]
	\centering
	\begin{tabular}{cc}
		%\toprule
		%ODCT & Gauss \\
		%\midrule
		%\includegraphics[scale=.32]{APG_CN_figure/perf_profile_cup_ODCT.eps} &
		%\includegraphics[scale=.32]{APG_CN_figure/perf_profile_cup_Gauss.eps} \\
		\includegraphics[scale=.4]{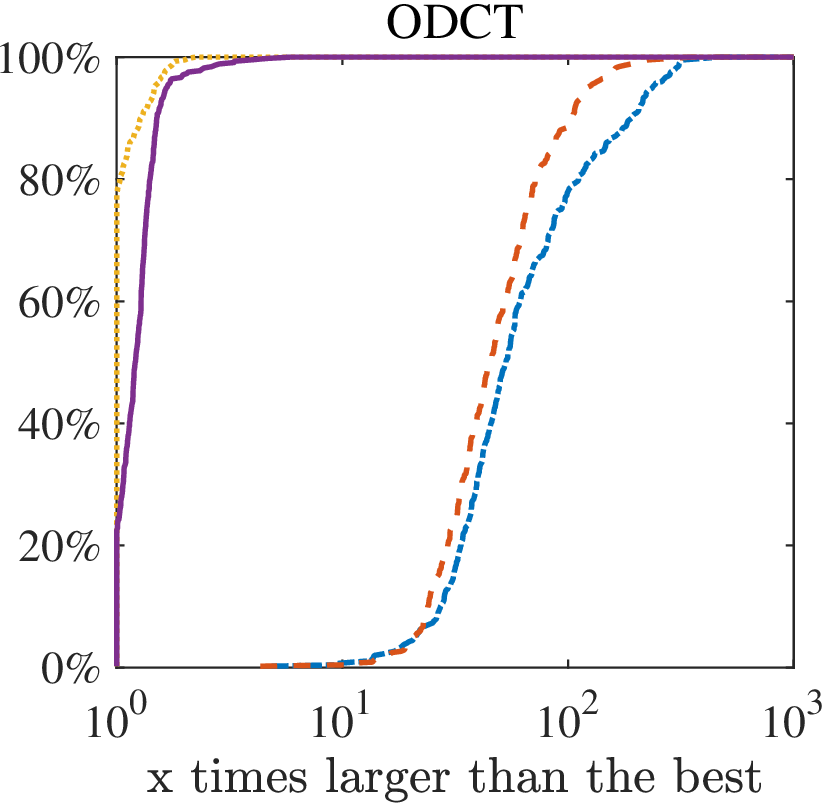} &
		\includegraphics[scale=.4]{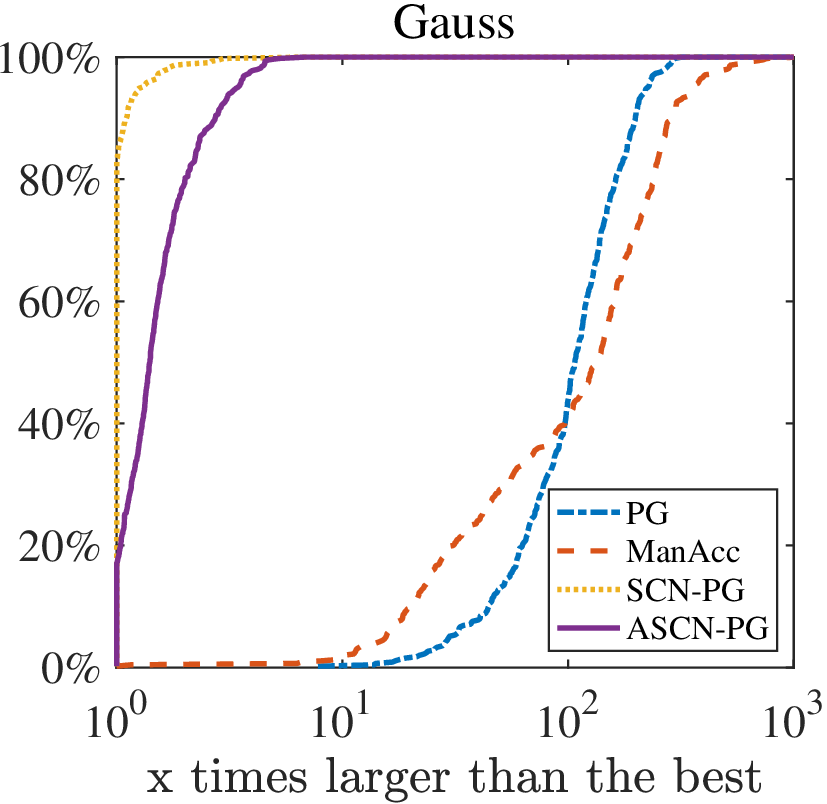} \\
		%\bottomrule
	\end{tabular}
	\caption{The performance profile of CPU.}
	\label{fig:performance_profile}
\end{figure}

%	\begin{figure}[t]
	%		\centering
	%		\caption*{\textbf{Comparison of Objective, Error and KKT versus Iteration ($\times 10^4$)}}%
	%		\vspace{1mm}
	%		\begin{subfigure}{1}
		%			\includegraphics[width=0.1\linewidth]{APG_CN_figure/ODCT_Iteration_obj.eps}
		%			\caption{Obj vs Iter $(\times 10^4)$}
		%		\end{subfigure}
	%		\begin{subfigure}{2}
		%			\includegraphics[width=0.1\linewidth]{APG_CN_figure/ODCT_Iteration_obj.eps}
		%			\caption{Obj vs Iter $(\times 10^4)$}
		%		\end{subfigure}
	%	\begin{subfigure}{3}
		%		\includegraphics[width=0.1\linewidth]{APG_CN_figure/ODCT_Iteration_obj.eps}
		%		\caption{Obj vs Iter $(\times 10^4)$}
		%	\end{subfigure}
	%		
	%		\vspace{-1mm}
	%		\caption{$r_k$ and $\text{rank}(X_k)$ for RA-PFBS-L$_1$/L$_2$, RA-PFBS-MCP, RA-PFBS-SCAD and RA-PFBS-Lp.}
	%		\label{fig:odct_iterations}
	%	\end{figure}
Next, we compare the numerical performance of PG, ManAcc, SCN-PG and ASCN-PG.
Figure \ref{fig:emperical_CDF_kkt} displays the empirical cumulative distribution of the KKT residuals obtained by PG, ManAcc, SCN-PG and ASCN-PG
% when test 900 instances with solving two types of matrix: 
 when solving 900 test instances generated from the two types of sensing matrices:
 ODCT and Gauss, $F=10,15,20$ and $K=2,4,6,8,10$ and randomly generated 30
instances with different initial points and randomly generated two types matrices.
Specifically, a point $(x,y)$ on the curve indicates that the proportion of 900 test instances for which the KKT residual of the solution obtained by the algorithm is smaller than $x$ equals $y$. As shown in the figure,  in approximately 90\% of the cases, the KKT residuals produced by SCN-PG andASCN-PG are below $10^{-15}$, whereas those of ManAcc and PG are around $10^{-9}$ and $10^{-4}$, respectively. This demonstrates the significantly higher solution accuracy achieved by SCN-PG and ASCN-PG.

Figure \ref{fig:performance_profile} presents the performance profile for all tested algorithms for the same
450 test instances for ODCT (left) and Gauss (right), respectively. In each plot, the point where the `percentage' axis intersects with the curve indicates the performance ratio of the current solver compared to the best solver. From the figures, we observe that in approximately $80\%$ of the cases, {SCN-PG} requires the least CPU time, whereas in about $20\%$ of the cases, {ASCN-PG} achieves the smallest CPU time. Moreover, ManAcc and PG are approximately $50$ times slower than SCN-PG in nearly $50\%$ of the test instance.
%		Regarding the number of iterations, ManAcc and PG require roughly $10$ times and $1000$ times more iterations than ASCN-PG in nearly half of the experiments.
These observations clearly demonstrate the superior efficiency of SCN-PG and ASCN-PG.

\subsection{Logistic Regression}\label{sec:logistic}

We solve the following $L_1 - \mu L_2$-regularized logistic regression problem \citep{qin20191}:
\begin{equation}
	\min_{(\z,u)\in\mathbb{R}^{n}\times\mathbb{R}} \frac{1}{m}\sum_{i = 1}^m\log(1 + e^{-b_i(\mathbf{a}^{\top}_i\z+u)}) + \lambda(\|\z\|_1 - \mu \|\z\|),\label{prob:L1mL2_logistic}
\end{equation}
where $\lambda > 0$, and feature vectors $\mathbf{a}_i\in\mathbb{R}^n,i\in[m]$,   and the labels $b_i\in\{-1,+1\},i\in[m]$. Let $A = [\mathbf{a}_1,\mathbf{a}_2,...,\mathbf{a}_m]$, $\mathbf{b} = [b_1,b_2,...,b_m]^{\top}, \x = (\z,u)$ and define $\phi(\x;A,\mathbf{b}) = \sum_{i = 1}^m\log(1 + e^{-b_i(\mathbf{a}^{\top}_i\z+u)})/m$. We will compare SCN-PG, ASCN-PG with PG and ManAcc on both synthetic datasets and real-world datasets.
In model (\ref{prob:L1mL2_logistic}), we set $\mu=0.5, 1$ and $\lambda=5\times 10^{-2}$ for {\it Synthetic} data.
For all these tested algorithms, we set $h(\x) = \phi(\x;A,\mathbf{b}),\ g(\x)=\lambda(\|\z\|_1 - \mu \|\z\|)$.
To apply PG, we take $\alpha=0.1$. 
To apply ManAcc, we take  $\alpha=0.1$ and $\gamma = 0.1, m_1 = 10^{-4}, \theta = 0.5, \mathcal{M}_k = \{\mathbf{w}: \supp(\mathbf{w}) = \supp(\x^k)\}$. Moreover, to enhance the convergence, we also use $H^k + \eta I$ as the modified Hessian matrix at each iteration and we set $\eta = 10^{-3}$.
To apply SCN-PG, we take  $\alpha=0.1$ and $\alpha = 0.1, c= 10^{-3},\eta = 10^{-10}, q= 0.95, M(\x^k) = 1.1(M^k_{\mathrm{log}} + {6\lambda\mu}/{\|\x^k\|^{2}}),\ \sigma_k = 0.67(M(\x^k) + c)$ and $\varepsilon_k = (\tau_1\|(\x^{k} - \y^{k-1})/\alpha\|+ \tau_2\|(\x^{k-1} - \y^{k-1})/\alpha\|)^{\tau_3}$ with $\tau_1 = 0.99,\ \tau_2 = 0.01,\ \tau_3 = 0.9$, where $M^k_{\mathrm{log}} :=\sum_{i=1}^m (\|\mathbf{a}_{i,{\cal S}^k}\|_2^2+1)^{3/2}/(6\sqrt{3}m).$ 
To apply ASCN-PG, we take  $\alpha = 0.1,\ c= 10^{-3},\ \eta = 10^{-10},\  q= 0.95,\ \eta_1 = 0.1,\ \eta_2 = 0.9,\ \gamma_1 = 0.5,\ \gamma_2 = 2,\ \sigma_{\min} = 2c,\ \sigma_{\max} = 100\sigma_{\min}$ and $\varepsilon_k = (\tau_1\|(\x^{k} - \y^{k-1})/\alpha\|+ \tau_2\|(\x^{k-1} - \y^{k-1})/\alpha\|)^{\tau_3}$ with $\tau_1 = 0.99,\ \tau_2 = 0.01,\ \tau_3 = 0.9$.

\subsubsection{Synthetic Data}
We generate the data using the method described in \citep{koh2007interior}. Specifically, we consider problem sizes $(m,n) \in \{(10i,100i): i = 1,2,\ldots,20\}$. For each pair $(m,n)$, we generate $m/2$ positive examples and $m/2$ negative examples.
The features of the positive (negative) examples are independent and identically distributed, drawn from a normal distribution $\mathcal{N}(v,1)$, where $v$ is  drawn from a uniform distribution on $[0,1]$ (respectively, $[-1,0]$).
The initial point is randomly generated from standard Gauss distribution.
For termination criterion (\ref{stoppping}), we set $\texttt{RTol} = 10^{-8},\texttt{CTol} = 50$ and $\texttt{Maxit} = 200000$. To measure the classification accuracy, we record Acc defined by:
$
\text{Acc} = {\sharp\{i : \textrm{sign}(\mathbf{a}_i^{\top}\z + u) = b_i\}\rvert}/{m}.
$

\begin{figure}[t]
	\centering
	\begin{tabular}{cc}
		%\toprule
		%ODCT & Gauss \\
		%\midrule
		%\includegraphics[scale=.32]{APG_CN_figure/CPU_logistic_mu05.eps} &
		%\includegraphics[scale=.32]{APG_CN_figure/CPU_logistic_mu1.eps} \\
		\includegraphics[scale=.4]{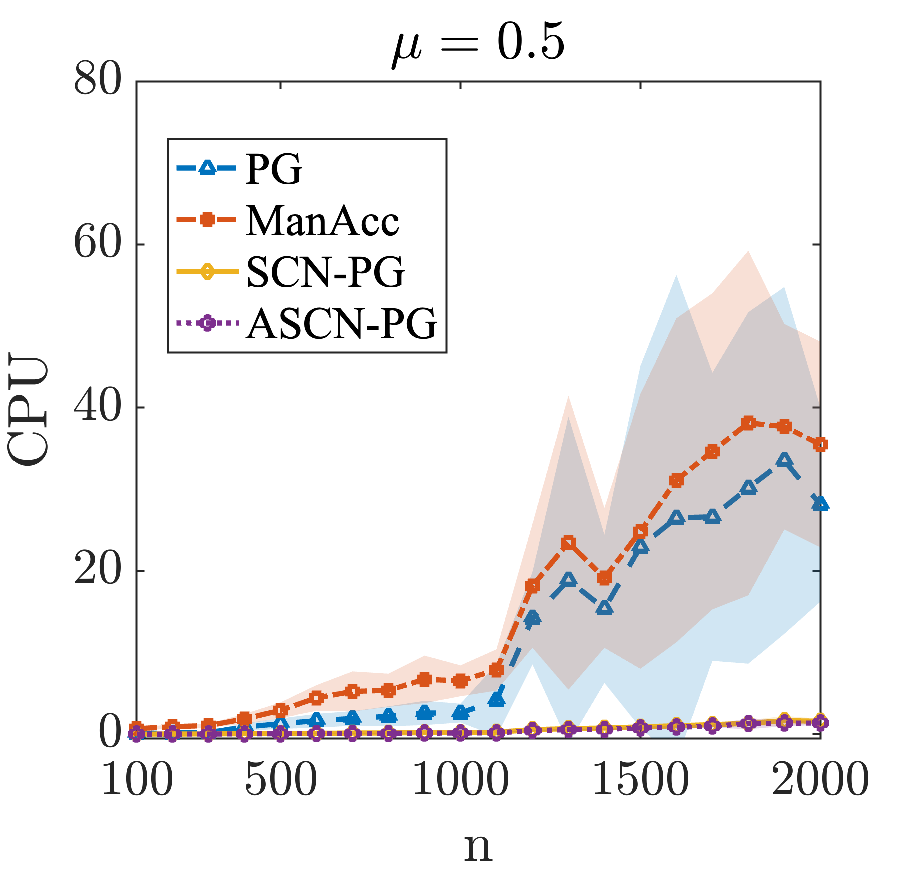} &
		\includegraphics[scale=.4]{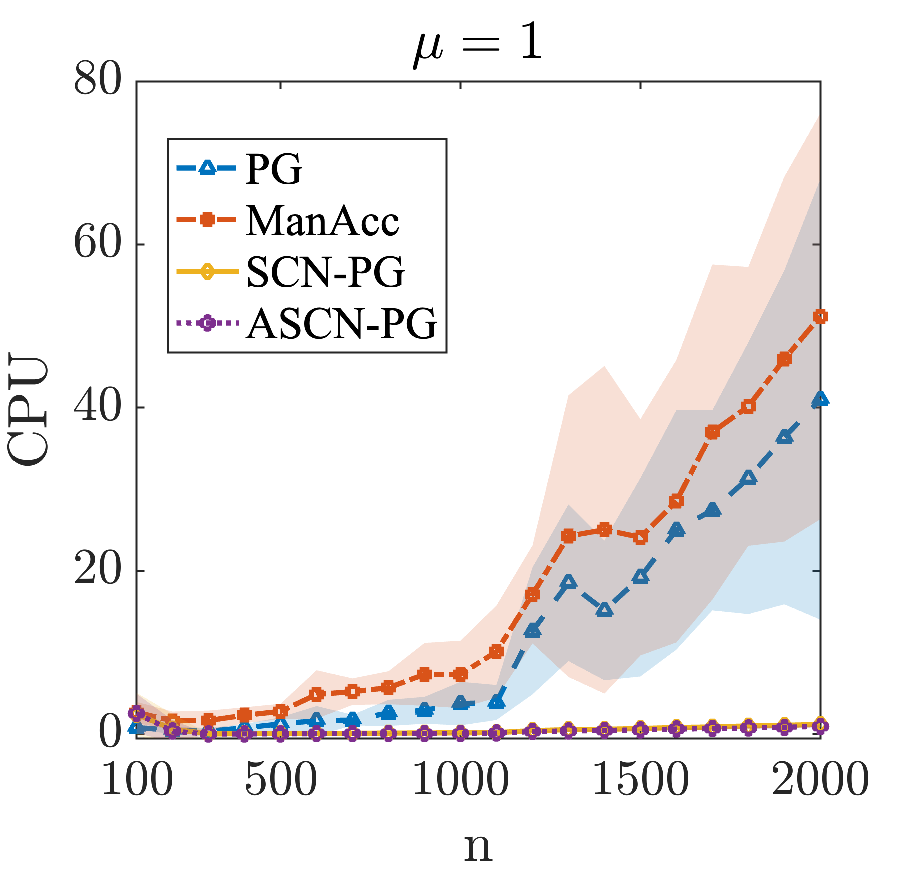} \\
		%\bottomrule
	\end{tabular}
	\caption{CPU time of PG, ManAcc, and ASCN-PG w.r.t.  $n$ for $\mu=0.5$ and $\mu=1$.}
	\label{fig:CPU_logistic}
\end{figure}
Figure \ref{fig:CPU_logistic} shows the evolution of the CPU time consumed by PG, ManAcc, {SCN-PG} and ASCN-PG as $n$ 
from $100$ to $2000$ for $\mu = 0.5$ and $\mu = 1$. The curves represent the averaged CPU time, while the shaded regions indicate the intervals corresponding to one standard deviation above and below the mean value. As observed, when $n$ is small, the CPU times of the three methods are comparable. As $n$ increases,  the CPU times of PG and ManAcc grow significantly than that of {SCN-PG} and ASCN-PG. In addition, the variance associated with {SCN-PG} and ASCN-PG is considerably smaller than that of PG and ManAcc, which further demonstrates the stability of them.

\begin{table}[htpb]
	\centering
	\small
	\setlength{\tabcolsep}{6pt}
	\renewcommand{\arraystretch}{1.15}
	\caption{Averaged results from the PG, ManAcc, SCN-PG and ASCN-PG for different  $\mu$ and $n$.}
	\begin{tabular}{cccccccc}
		\toprule
		$\mu = 0.5$ & Algorithm & Obj & NNZ & Acc & KKT & CPU & Iter \\
		\midrule
		\multirow{4}{*}{$n=400$} & PG & 2.09e-01 & 9.63 & 1.00 & 1.79e-07 & 8.72e-01 & 21584.30 \\
		& ManAcc & 2.10e-01 & 9.77 & 1.00 & 1.21e-07 & 1.87e+00 & 13903.67 \\
		& SCN-PG & 2.09e-01 & 9.73 & 1.00 & 1.10e-16 & 6.13e-02 & 462.93 \\
		& ASCN-PG & 2.10e-01 & 9.80 & 1.00 & 1.67e-16 & 3.93e-02 & 380.77 \\
		\midrule
		\multirow{4}{*}{$n=800$} & PG & 2.43e-01 & 16.03 & 1.00 & 1.36e-07 & 2.19e+00 & 24576.57 \\
		& ManAcc & 2.43e-01 & 16.13 & 1.00 & 1.07e-07 & 5.38e+00 & 20578.07 \\
		& SCN-PG & 2.43e-01 & 16.10 & 1.00 & 2.49e-16 & 1.53e-01 & 670.97 \\
		& ASCN-PG & 2.43e-01 & 16.23 & 1.00 & 3.71e-16 & 1.05e-01 & 539.83 \\
		\midrule
		\multirow{4}{*}{$n=1200$} & PG & 2.61e-01 & 21.97 & 1.00 & 1.19e-07 & 1.42e+01 & 30739.17 \\
		& ManAcc & 2.61e-01 & 21.97 & 1.00 & 1.01e-07 & 1.82e+01 & 25740.53 \\
		& SCN-PG & 2.61e-01 & 21.97 & 1.00 & 2.60e-16 & 5.53e-01 & 863.07 \\
		& ASCN-PG & 2.61e-01 & 22.03 & 1.00 & 2.03e-16 & 4.37e-01 & 733.03 \\
		\midrule
		\multirow{4}{*}{$n=1600$} & PG & 2.71e-01 & 26.03 & 1.00 & 1.08e-07 & 2.65e+01 & 30556.23 \\
		& ManAcc & 2.71e-01 & 26.07 & 1.00 & 1.47e-06 & 3.11e+01 & 26610.20 \\
		& SCN-PG & 2.71e-01 & 26.10 & 1.00 & 1.74e-16 & 1.03e+00 & 1003.73 \\
		& ASCN-PG & 2.71e-01 & 25.97 & 1.00 & 3.00e-16 & 8.55e-01 & 869.73 \\
		\midrule
		\multirow{4}{*}{$n=2000$} & PG & 2.78e-01 & 29.83 & 1.00 & 9.94e-08 & 2.81e+01 & 23053.03 \\
		& ManAcc & 2.78e-01 & 29.90 & 1.00 & 9.02e-08 & 3.55e+01 & 20287.17 \\
		& SCN-PG & 2.78e-01 & 29.93 & 1.00 & 1.95e-16 & 1.59e+00 & 1128.07 \\
		& ASCN-PG & 2.78e-01 & 29.83 & 1.00 & 1.92e-16 & 1.37e+00 & 1006.67 \\
		\midrule
		$\mu = 1$ & Algorithm & Obj & NNZ & Acc & KKT & CPU & Iter \\
		\midrule
		\multirow{4}{*}{$n=400$} & PG & 1.54e-01 & 5.70 & 1.00 & 3.18e-07 & 7.76e-01 & 21072.23 \\
		& ManAcc & 1.52e-01 & 5.77 & 1.00 & 2.63e-07 & 2.32e+00 & 23687.30 \\
		& SCN-PG & 1.60e-01 & 6.17 & 1.00 & 2.61e-16 & 5.24e-02 & 413.47 \\
		& ASCN-PG & 1.59e-01 & 6.03 & 1.00 & 1.31e-16 & 2.71e-02 & 322.30 \\
		\midrule
		\multirow{4}{*}{$n=800$} & PG & 2.02e-01 & 9.87 & 1.00 & 2.23e-07 & 2.61e+00 & 31252.03 \\
		& ManAcc & 2.02e-01 & 10.03 & 1.00 & 1.93e-07 & 5.69e+00 & 24445.63 \\
		& SCN-PG & 2.02e-01 & 9.77 & 1.00 & 1.54e-16 & 1.28e-01 & 605.60 \\
		& ASCN-PG & 2.01e-01 & 9.97 & 1.00 & 7.73e-17 & 7.62e-02 & 469.20 \\
		\midrule
		\multirow{4}{*}{$n=1200$} & PG & 2.24e-01 & 13.33 & 1.00 & 1.93e-07 & 1.26e+01 & 28834.63 \\
		& ManAcc & 2.24e-01 & 13.33 & 1.00 & 1.74e-07 & 1.71e+01 & 26889.73 \\
		& SCN-PG & 2.25e-01 & 13.33 & 1.00 & 2.24e-16 & 4.35e-01 & 742.80 \\
		& ASCN-PG & 2.25e-01 & 13.30 & 1.00 & 1.41e-16 & 3.16e-01 & 596.73 \\
		\midrule
		\multirow{4}{*}{$n=1600$} & PG & 2.38e-01 & 15.03 & 1.00 & 1.67e-07 & 2.50e+01 & 31944.07 \\
		& ManAcc & 2.38e-01 & 15.30 & 1.00 & 2.46e-07 & 2.86e+01 & 26492.20 \\
		& SCN-PG & 2.38e-01 & 15.40 & 1.00 & 9.43e-17 & 8.19e-01 & 876.60 \\
		& ASCN-PG & 2.38e-01 & 14.93 & 1.00 & 1.70e-16 & 6.24e-01 & 702.43 \\
		\midrule
		\multirow{4}{*}{$n=2000$} & PG & 2.47e-01 & 17.47 & 1.00 & 1.54e-07 & 4.10e+01 & 35403.33 \\
		& ManAcc & 2.48e-01 & 17.73 & 1.00 & 1.44e-07 & 5.12e+01 & 35251.10 \\
		& SCN-PG & 2.48e-01 & 17.50 & 1.00 & 1.52e-16 & 1.19e+00 & 930.53 \\
		& ASCN-PG & 2.48e-01 & 17.77 & 1.00 & 2.34e-16 & 9.57e-01 & 771.50 \\
		\bottomrule
	\end{tabular}
	
	\label{tab:mean_logistic}
\end{table}

Table \ref{tab:mean_logistic} reports  objective function value (Obj), number of nonzero element (NNZ), the CPU time (CPU), the Iteration numbers (Iter), 
and the KKT residual (KKT), the classification accuracy (Acc)
obtained by PG, ManAcc, {SCN-PG} and ASCN-PG for $\mu \in \{0.5,1\}$ and $n \in \{400,800,1200,1600,2000\}$.
From the table, we observe that the solutions produced by the three methods have very similar values of {Obj}, {NNZ}, and {Acc}. However, {SCN-PG} and ASCN-PG achieves significantly smaller {KKT} residuals and requires substantially less CPU time and fewer iterations than PG and ManAcc. In particular, the {KKT} residuals of PG and ManAcc are on the order of $10^{-8}$, whereas {SCN-PG} and ASCN-PG attains residuals on the order of $10^{-16}$. Moreover, the CPU time consumed by SCN-PG and ASCN-PG are only about $5\%$ and $2\%$ of that required by the other methods, respectively. These results clearly demonstrate the high computational efficiency of {SCN-PG} and ASCN-PG.

\subsubsection{Real-World Data}
We tested on a real-world dataset.  Consider two datasets, each of them consists a training set $(A,\mathbf{b})\in\mathbb{R}^{m\times n}\times\mathbb{R}^{m}$ and a testing set $(A_{\textrm{test}},\mathbf{b}_{\textrm{test}})\in\mathbb{R}^{m_{\textrm{te}}\times n}\times\mathbb{R}^{m_{\textrm{te}}}$.
\begin{itemize}
	\item \texttt{a9a}\footnote{\url{https://www.csie.ntu.edu.tw/~cjlin/libsvmtools/datasets/binary.html}}:
	Binary classification dataset derived from the Adult dataset in the UCI Machine Learning Repository. The task is to predict whether an individual's annual income exceeds \$50,000 based on census attributes. After preprocessing, each sample is represented by a sparse feature vector with $n=123$ features. The dataset contains $m = 32561$ training samples and $m_{\textrm{te}} = 16281$ test samples.
	\item \texttt{MNIST}\footnote{\url{http://yann.lecun.com/exdb/mnist/}}: Handwritten digit recognition. It consists of grayscale images of handwritten digits from $0$ to $9$, where each image has a resolution of 28 $\times$ 28 pixels. After vectorization, each sample is represented by a feature vector of dimension $n = 784$. The dataset contains $m = 60000$ training samples and $m_{\textrm{te}} = 10000$ test samples. We consider a binary classification task by grouping even digits as the positive class and odd digits as the negative class.
\end{itemize}

%To apply PG, we take $h(\x,u) = \phi(\x,u;A_{\textrm{train}},\mathbf{b}_{\textrm{train}}), g(\x)=\lambda(\|\x\|_1 - \mu \|\x\|)$, where $\lambda = 5\times10^{-2}$. We set $\alpha = 1$.
%
%To apply ManAcc, we take $f(\x,u) = \phi(\x,u;A_{\textrm{train}},\mathbf{b}_{\textrm{train}}), g(\x)=\lambda(\|\x\|_1 - \mu \|\x\|)$, where $5\times10^{-2}$. We set $\gamma = 1, m_1 = 10^{-4}, \theta = 0.5, \mathcal{M}_k = \{\z: \supp(\z) = \supp(\x^k)\}$. Moreover, to enhance the convergence, we also use $H^k + \eta I$ as the modified Hessian matrix at each iteration and we set $\eta = 10^{-3}$.
%
%To apply ASCN-PG, we take $h(\x,u) = \phi(\x,u;A_{\textrm{train}},\mathbf{b}_{\textrm{train}}), g(\x)=\lambda(\|\x\|_1 - \mu \|\x\|)$, where $\lambda = 5\times10^{-2}$. We set $\alpha = 1, c= 10^{-3},\eta = 10^{-10}, q= 0.95, \eta_1 = 0.1,\eta_2 = 0.9, \gamma_1 = 0.5,\gamma_2 = 2,\sigma_{\min} = 2c,\sigma_{\max} = 100\sigma_{\min}$ and $\varepsilon_k = \tau_1\|(\x^{k} - \y^{k-1})/\alpha\|^{\tau_2}$ with $\tau_1 = 0.85,\tau_2 = 0.99$.

The initial point is randomly generated from standard Gauss distribution.
For termination criterion (\ref{stoppping}), we set $\texttt{RTol} = 10^{-6},\texttt{CTol} = 1000$ and $\texttt{Maxit} = 50000$. 
To measure the classification accuracy, we  compute the testing accuracy (TEAcc) defined by:
$$\textrm{TEAcc} = \displaystyle{\frac{\sharp\{i : \textrm{sign}(\mathbf{a}_{\textrm{test},i}^{\top}\z + u) = b_{\textrm{test},i}\}}{m_{\text{te}}}}.
$$

%In model (\ref{prob:L1mL2_logistic}), we set $(\mu,\lambda)=(0.5,5e-2)$ and $(\mu,\lambda)=(1.0,5e-2)$  for \texttt{a9a};
%and  $(\mu,\lambda)=(0.5,2e-2)$ and $(\mu,\lambda)=(1.0,2e-2)$  for \texttt{MNIST} by hand-tuning.
In model \eqref{prob:L1mL2_logistic}, we set $(\mu,\lambda)=(0.5,5\times10^{-2})$ and $(\mu,\lambda)=(1.0,5\times10^{-2})$ for the \texttt{a9a} dataset, and $(\mu,\lambda)=(0.5,2\times10^{-2})$ and $(\mu,\lambda)=(1.0,2\times10^{-2})$ for the \texttt{MNIST} dataset. 
These parameters are selected via manual tuning.
The remaining algorithmic parameters are chosen as follows. 
For the \texttt{a9a} dataset with $(\mu,\lambda)=(0.5,5\times10^{-2})$, we adopt the same settings as in Section~\ref{sec:logistic}. 
For $(\mu,\lambda)=(1.0,5\times10^{-2})$, we use the same settings except that $\alpha=1$ for all algorithms. 
For the \texttt{MNIST} dataset, we set $\alpha=1$ for all algorithms across all tested cases.

\begin{figure}[htpb]
	\centering
	\begin{tabular}{cc}
		%\toprule
		%ODCT & Gauss \\
		%\midrule
		%\includegraphics[scale=.37]{APG_CN_figure/KKT_a9a_mu055.eps} 
		%& \includegraphics[scale=.37]{APG_CN_figure/KKT_MNIST_mu055.eps} \\
		%\includegraphics[scale=.32]{APG_CN_figure/KKT_a9a_mu1.eps} &
		%\includegraphics[scale=.32]{APG_CN_figure/KKT_MNIST_mu1.eps}
		\includegraphics[scale=.4]{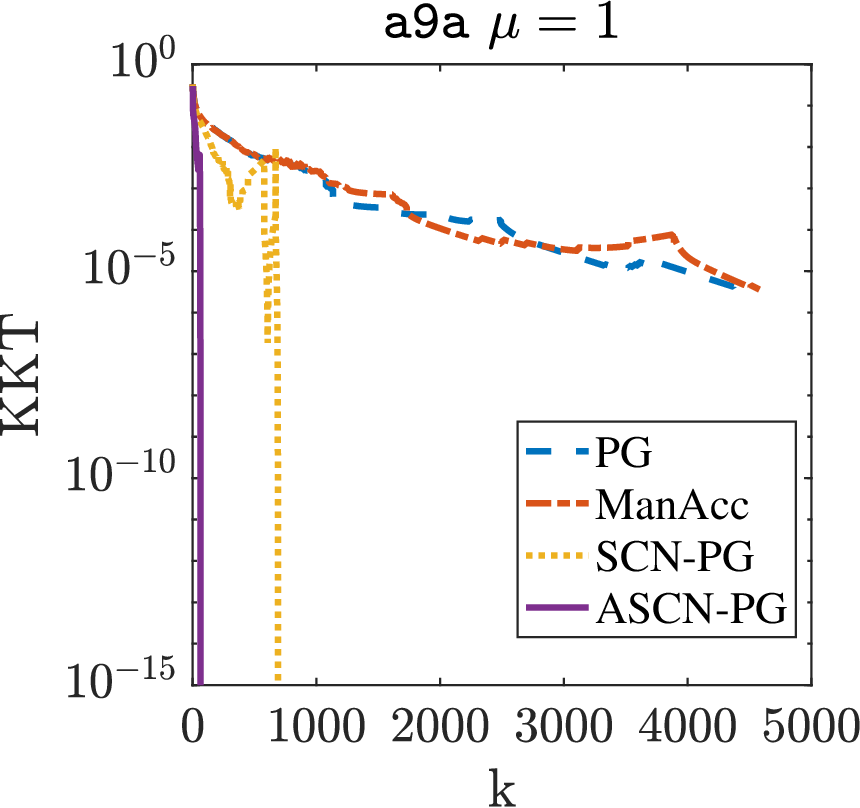} 
		& \includegraphics[scale=.4]{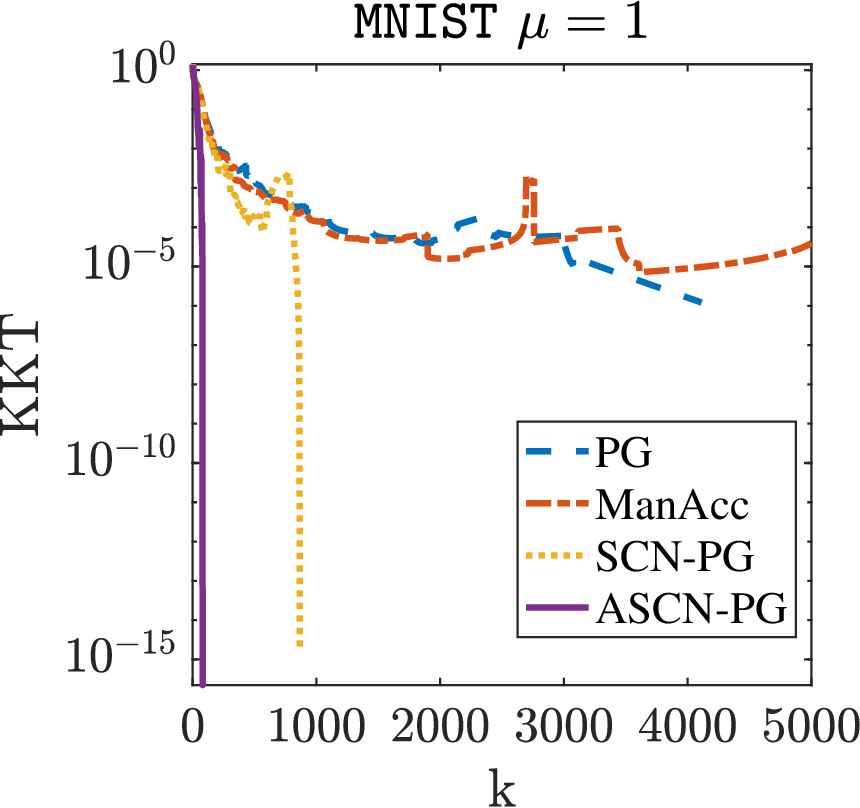}
		%\bottomrule
	\end{tabular}
	\caption{The Evolution of the KKT residual w.r.t. the iteration number.}
	\label{fig:KKT_realdata}
\end{figure}
Figure~\ref{fig:KKT_realdata} plots the KKT residual versus the iteration number for the four methods with $\mu=1$ on the \texttt{MNIST} and \texttt{a9a} datasets. 
It is evident that ASCN-PG and SCN-PG achieve a substantially faster reduction in the KKT residual than PG and ManAcc. 
In particular, the KKT residual of ASCN-PG decreases to the order of $10^{-16}$ within fewer than $100$ iterations, while SCN-PG attains a residual in the range $10^{-13}$--$10^{-15}$ after $400$--$500$ iterations. 
In contrast, PG and ManAcc stagnate around $10^{-6}$ even after $2000$--$4000$ iterations. 
These results indicate that ASCN-PG is more computationally efficient than SCN-PG in practice.
Table~\ref{tab:real_data_binary} reports the performance of all methods in terms of {Obj}, {NNZ}, {TEAcc}, {KKT}, {CPU}, and {Iter} for $\mu \in \{0.5,1\}$. 
All methods produce nearly identical values of {Obj}, {NNZ}, and {TEAcc}, whereas ASCN-PG achieves significantly smaller {KKT} residuals and markedly lower CPU time. 
On both the \texttt{a9a} and \texttt{MNIST} datasets, the CPU time of ASCN-PG is approximately $5\%$ and $0.5\%$ of that required by PG and ManAcc, respectively. 
Moreover, ASCN-PG consistently outperforms SCN-PG in terms of computational efficiency. 
This advantage is mainly due to the conservative estimate of $M(\x^k)$ in SCN-PG, which leads to smaller steps in the cubic Newton update.
Overall, these results demonstrate that SCN-PG and ASCN-PG attain high-accuracy solutions at a substantially lower computational cost than PG and ManAcc, highlighting the effectiveness of the proposed methods on both synthetic and real datasets.
\begin{table}[htpb]
	\centering
	\small
	\setlength{\tabcolsep}{6pt}
	\renewcommand{\arraystretch}{1.15}
	\caption{Results from PG, ManAcc, SCN-PG and ASCN-PG on two datasets.}
	\begin{tabular}{cc c c c c c r}
		\toprule
		$\texttt{a9a}$ & Algorithm & Obj & NNZ & TEAcc & KKT & CPU & Iter \\
		\midrule
		& PG & 3.87e-01 & 15.10 & 0.84 & 3.19e-05 & 2.23e+01 & 3858.80 \\
		$\mu = 0.5$ & ManAcc & 3.87e-01 & 15.00 & 0.84 & 3.81e-06 & 1.34e+02 & 2794.60 \\
		& SCN-PG & 3.87e-01 & 15.00 & 0.84 & 7.59e-15 & 6.66e+00 & 384.70 \\
		& ASCN-PG & 3.87e-01 & 15.00 & 0.84 & 2.66e-16 & 9.22e-01 & 53.40 \\
		%& PG & 3.87e-01 & 15.10 & 0.84 & 3.19e-05 & 2.22e+01 & 3858.80 \\
		%$\mu = 0.5$& ManAcc & 3.87e-01 & 15.00 & 0.84 & 3.81e-06 & 1.64e+02 & 2794.60 \\
		%& ASCN-PG & 3.87e-01 & 15.00 & 0.84 & 6.78e-17 & 9.91e-01 & 50.30 \\
		\toprule
		%			$\texttt{a9a}$  & Solver & Obj & NNZ & Acc & KKT & CPU & Iter \\
		%			\midrule
		& PG & 3.79e-01 & 14.10 & 0.84 & 3.55e-06 & 2.09e+01 & 3565.50 \\
		$\mu = 1$& ManAcc & 3.79e-01 & 14.00 & 0.84 & 3.98e-06 & 1.82e+02 & 3194.60 \\
		& SCN-PG & 3.79e-01 & 14.00 & 0.84 & 1.38e-15 & 6.98e+00 & 402.50 \\
		& ASCN-PG & 3.79e-01 & 14.00 & 0.84 & 1.31e-16 & 9.30e-01 & 54.60 \\
		%& PG & 3.79e-01 & 14.10 & 0.84 & 3.55e-06 & 2.14e+01 & 3565.50 \\
		%$\mu = 1$& ManAcc & 3.79e-01 & 14.00 & 0.84 & 3.98e-06 & 2.28e+02 & 3194.60 \\
		%& ASCN-PG & 3.79e-01 & 14.00 & 0.84 & 2.09e-16 & 1.00e+00 & 51.90 \\
		\toprule
		\texttt{MNIST} & Algorithm & Obj & NNZ & TEAcc & KKT & CPU & Iter \\
		\midrule
		& PG & 4.34e-01 & 57.30 & 0.87 & 5.04e-06 & 4.69e+02 & 4228.70 \\
		$\mu = 0.5$& ManAcc & 4.35e-01 & 59.30 & 0.87 & 1.94e-03 & 8.99e+02 & 3577.70 \\
		& SCN-PG & 4.34e-01 & 57.00 & 0.87 & 3.75e-12 & 1.34e+02 & 813.80 \\
		& ASCN-PG & 4.34e-01 & 57.10 & 0.87 & 2.53e-16 & 1.73e+01 & 92.40 \\
		%& PG & 4.34e-01 & 57.30 & 0.87 & 5.04e-06 & 4.68e+02 & 4228.70 \\
		%$\mu = 0.5$ & ManAcc & 4.35e-01 & 59.50 & 0.87 & 1.78e-03 & 9.17e+02 & 3477.20 \\
		%& ASCN-PG & 4.34e-01 & 57.60 & 0.87 & 5.31e-16 & 1.47e+01 & 65.90 \\
		\midrule
		& PG & 4.23e-01 & 41.60 & 0.87 & 1.20e-06 & 2.66e+02 & 2348.20 \\
		$\mu = 1$& ManAcc & 4.27e-01 & 48.80 & 0.87 & 6.96e-04 & 8.80e+02 & 2523.80 \\
		& SCN-PG & 4.23e-01 & 41.70 & 0.87 & 1.23e-13 & 8.60e+01 & 512.10 \\
		& ASCN-PG & 4.23e-01 & 42.00 & 0.87 & 1.98e-16 & 1.33e+01 & 74.00 \\
		%\bottomrule
		%
		%& PG & 4.23e-01 & 41.60 & 0.87 & 1.20e-06 & 2.62e+02 & 2348.20 \\
		%$\mu = 1$& ManAcc & 4.27e-01 & 48.80 & 0.87 & 6.82e-04 & 8.80e+02 & 2510.00 \\
		%& ASCN-PG & 4.23e-01 & 42.10 & 0.87 & 2.56e-16 & 1.38e+01 & 64.30 \\
		\bottomrule
	\end{tabular}
	\label{tab:real_data_binary}
\end{table}

\section{Conclusions}\label{sec:conclusion}

This paper studies a class of nonsmooth and nonconvex composite optimization problems and proposes an alternating minimization framework that integrates proximal-gradient steps with subspace cubic-regularized Newton updates. The proposed method achieves global convergence to a stationary point under the KL property.
By employing an adaptive thresholding strategy guided by the KL exponent, the framework unifies active manifold identification and second-order acceleration without requiring partly smoothness, prox-regularity, or nondegeneracy conditions. This design enables both finite identification and efficient local convergence.
The framework is flexible and admits several extensions. In particular, the subspace cubic Newton step can be replaced by other second-order acceleration schemes, such as proximal Newton-type methods, while preserving global convergence guarantees. In this work, we focus on a simple subspace cubic Newton strategy to highlight the core ideas underlying the proposed framework.

\vspace{0.2cm}
\noindent{\bf Funding} The work of M. Tao is supported by the Natural Science Foundation of China (No.
12471289,\ 12371318).

\noindent{\bf Data Availibility}
The datasets generated during and/or analysed during the current study are available from
the corresponding author on reasonable request.\\
%\vspace{0.2cm}
\noindent{\bf Declarations}
Conflict of interests. The authors have no relevant financial or non-financial interests to disclose.

\newpage
\appendix
\section{Estimation of the Locally Hessian Lipschitz Continuous Modulus}\label{app:LHLC}
%$$\bigtriangledown^3 f$$
\begin{theorem}
Let $F(\x) = \lambda\frac{\|\x\|_1}{\|\x\|_2} +\frac{1}{2}\|A\x-\mathbf{b}\|_{2}^{2}$, where $A\in\mathbb{R}^{m\times n},\mathbf{b}\in\mathbb{R}^{m}$ and $\lambda > 0$. Given a point $0\neq \x \in\mathbb{R}^n$ and define $\cal S = \supp(\x),\mathcal{M}_{\cal S} = \{\z: \supp(\z) = \supp(\x)\}$, then the following holds:
\begin{eqnarray*}
\mathrm{lip}\nabla^2_{\mathcal{S}}F({\x}) \le \frac{36\sqrt{\|\x\|_0}}{\|\x\|^3}.
\end{eqnarray*}
%thus for any $M(\x) > {36\sqrt{\|\x\|_0}}/{\|\x\|^3}$, there must be a $\delta > 0$, such that
%\begin{eqnarray*}
%\|\nabla^2_{\mathcal{S} }F(\x_1)-\nabla^2_{\mathcal{S} }F(\x_2)\|\le M(\x)\|\x_1-\x_2\|,
%\quad \forall \x_1,\x_2\in \cB(\x,\delta)\cap \mathcal{M}_{\cal S}.
%\end{eqnarray*}
\end{theorem}
\begin{proof}
Without loss of generality, we may assume that $\mathcal{S} = [n]$. Note that the restriction of $F$ to $\cal S$ is actually three times continuously differentiable, therefore by \cite[Theorem 9.7]{rockafellar1998variational} it holds that:
\begin{eqnarray*}
\mathrm{lip}\nabla^2_{\mathcal{S}}F({\x})= \|\nabla^3 F(\x)\|,
\end{eqnarray*}
and we only need to derive an upper bound of $ \|\nabla^3 F(\x)\|$. Simple calculations yield that for any $\mathbf{h}\in\mathbb{R}^n$ it holds:
\begin{equation*}
\begin{aligned}
\nabla^3F(\x)[\mathbf{h}] = &\;\frac{3\beta}{r^5}\big(\mathbf{a} \x^\top + \x \mathbf{a}^\top + u I\big)
-\frac{1}{r^3}\big(\mathbf{a} \mathbf{h}^\top + \mathrm{h} \mathbf{a}^\top + \alpha I\big) \\
&\;+\frac{3\alpha}{r^5}\x \x^\top
+\frac{3u}{r^5}\big(\x \mathbf{h}^\top + \mathbf{h} \x^\top\big)
-\frac{15u\beta}{r^7}\x \x^\top,
\end{aligned}
\end{equation*}
where $\mathbf{a} = \sign(\x), r = \|\x\|, \alpha = \mathbf{a}^\top \mathbf{h}, \beta = \x^\top \mathbf{h}$ and $ u = \mathbf{a}^\top \x$. Then we have 
\begin{equation*}
\|\nabla^3F(\x)[\mathbf{h}]\| \le (\frac{9}{r^3}+ \frac{3}{r^3}+\frac{9}{r^3}+\frac{15}{r^3})\sqrt{\|\x\|_0}\|\mathbf{h}\| = \frac{36\sqrt{\|\x\|_0}}{\|\x\|^3}\|\mathbf{h}\|,
\end{equation*}
which implies the results.

\end{proof}

\begin{theorem}
Let $F(\x)=\frac{1}{m}\sum_{i=1}^m \log\!\left(1+e^{-b_i(\mathbf{a}_i^\top \z+u)}\right)
+\lambda(\|\z\|_1-\mu\|\z\|_2),$
where $\x=(\z,u)\in\mathbb{R}^{n}\times\mathbb{R}$, $\mathbf{a}_i\in\mathbb{R}^n$, $b_i\in\{\pm 1\}$, and $\lambda,\mu>0$.
Given a point $\x=(\z,u)$ with $\x\neq 0$, define $\cal S=\supp(\x), \mathcal{M}_{\cal S}:=\{\mathbf{v}:\supp(\mathbf{v})= \supp(\mathbf{\x})\}$ and $\widehat{\cal S}=\supp(\z)$. Then the following holds:
\begin{eqnarray*}
\mathrm{lip}\nabla^2_{\mathcal{S}}F({\x})
\le 
M_{\mathrm{log}}+\frac{6\lambda\mu}{\|\z\|^{2}},
\end{eqnarray*}
where
\[
M_{\mathrm{log}}
:=\frac{1}{6\sqrt{3}m}\sum_{i=1}^m (\|\mathbf{a}_{i,\widehat{\cal S}}\|_2^2+1)^{3/2}.
\]
%Thus, for any $M(\x) > M_{\mathrm{log}}+\frac{6\lambda\mu}{\|\z\|^{2}}$, there exists a $\delta>0$ such that
%\[
%\|\nabla^2_{\cal S}F(\x_1)-\nabla^2_{\cal S}F(\x_2)\|
%\le M(\x)\|\x_1-\x_2\|,
%\quad \forall \x_1,\x_2\in \cB(\x,\delta)\cap \mathcal{M}_{\cal S}.
%\]
\end{theorem}

\begin{proof}
Without loss of generality, we may assume that $\widehat{\cal S}\neq\emptyset$ and $u \neq 0$.
Note that the restriction of $F$ to $\mathcal{M}_{\cal S}$ is three times continuously differentiable, therefore it suffices to bound $\|\nabla^3_{\cal S}F(\x)\|$.

First, for each $i\in[m]$ define the augmented vector $\mathbf{v}_i:=\binom{\mathbf{a}_{i,\widehat{\cal S}}}{1}\in\mathbb{R}^{s+1}$ and $\theta_i(\x):=b_i(\mathbf{a}_i^\top \z+u)=b_i \mathbf{v}_i^\top \binom{\z_{\widehat{\cal S}}}{u}.$
Let $\sigma(t)={1}/{(1+e^{-t})}$ and $w(t)=\sigma(t)(1-\sigma(t))$. Then a direct calculation yields for any $\mathbf{h}=(\mathbf{p},q)\in\mathbb{R}^{s}\times\mathbb{R}$,
\begin{equation*}
\begin{aligned}
\nabla^3_{\cal S} \log(1+e^{-\theta_i(\x)})[\mathbf{h}] &= w'(\theta_i(\x))\cdot b_i(\mathbf{v}_i^\top \mathbf{h})\, \mathbf{v}_i \mathbf{v}_i^\top .
\end{aligned}
\end{equation*}
Moreover, $\sup_{t\in\mathbb{R}}|w'(t)|=\sup_{t\in\mathbb{R}}|\sigma(t)(1-\sigma(t))(1-2\sigma(t))| = \frac{1}{6\sqrt{3}}$, thus
\begin{equation*}
\|\nabla^3_{\cal S} \log(1+e^{-\theta_i(\x)})[\mathbf{h}]\|
\le 
\frac{1}{6\sqrt{3}}\; |\mathbf{v}_i^\top \mathbf{h}|\, \|\mathbf{v}_i\|^2
\le
\frac{1}{6\sqrt{3}}\;\|\mathbf{v}_i\|^3 \|\mathbf{h}\|.
\end{equation*}
Therefore,
\begin{equation*}
\Big\|
\nabla^3_{\cal S}\Big(\frac{1}{m}\sum_{i=1}^m \log(1+e^{-\theta_i(\x)})\Big)[\mathbf{h}]
\Big\|
\le
L_{\mathrm{log}}\,\|\mathbf{h}\|.
\end{equation*}

Next, let $r:=\|\z_{\cal S}\|, \beta = \x^{\top}\mathbf{p}$. A direct calculation gives
\begin{equation*}
\begin{aligned}
(\nabla^3_{\cal S}\|\z_{\widehat{\cal S}}\|_2)[\mathbf{p}]= -\frac{\beta}{r^3}I +\frac{3\beta}{r^5}\z_{\cal S}\z_{\cal S}^\top -\frac{1}{r^3}(\z_{\cal S}\mathbf{p}^\top+\mathbf{p}\z_{\cal S}^\top).
\end{aligned}
\end{equation*}
Therefore, we obtain
\begin{equation*}
\|\nabla^3_{\widehat{\cal S}}(-\lambda\mu\|\z_{\widehat{\cal S}}\|)[\mathbf{h}]\|
\le
\frac{6\lambda\mu}{\|\z_{\widehat{\cal S}}\|^2}\,\|\mathbf{h}\|.
\end{equation*}
Combining the two parts and recalling that $\nabla^2_{\widehat{\cal S}}\|\z\|_1\equiv 0$ on $\mathcal{M}_{\widehat{\cal S}}$, we have
\[
\|\nabla^3_{{\cal S}}F(\x)[\mathbf{h}]\|
\le
\left(L_{\mathrm{log}}+\frac{6\lambda\mu}{\|\z_{\widehat{\cal S}}\|^2}\right)\|\mathbf{h}\|.
\]
which yields the desired bound.
\end{proof}

\end{document}